\documentclass[
hidempi,
]{mpi2015-cscpreprint}

\usepackage{manuscript}
\usepackage{orcidlink}

\begin{document}

\title{Generalizing Reduced Rank Extrapolation to Low-Rank Matrix Sequences}

\author[$\ast$]{Pascal~den~Boef\orcidlink{0009-0004-9061-9284}} % chktex 8
\author[$\dagger$]{Patrick~K\"urschner\orcidlink{0000-0002-6114-8821}} % chktex 8
\author[$\ddagger$]{Xiaobo~Liu\orcidlink{0000-0001-8470-8388}} % chktex 8
\author[$\ast$]{Joseph~Maubach\orcidlink{0009-0004-4538-0091}} % chktex 8
\author[$\ddagger$]{Jens~Saak\orcidlink{0000-0001-5567-9637}} % chktex 8
\author[$\ast$]{Wil~Schilders\orcidlink{0000-0001-5838-8579}} % chktex 8
\author[$\ddagger$]{Jonas~Schulze\orcidlink{0000-0002-2086-7686}} % chktex 8
\author[$\ast$]{Nathan~van~de~Wouw\orcidlink{0000-0002-6745-9137}} % chktex 8

\affil[$\ast$]{Eindhoven University of Technology, Eindhoven, Netherlands.}
\affil[$\dagger$]{Leipzig University of Applied Sciences, Leipzig, Germany.}
\affil[$\ddagger$]{Max Planck Institute for Dynamics of Complex Technical Systems, Magdeburg, Germany.}

\abstract{\Ac{RRE} is an acceleration method typically used to accelerate the iterative solution of nonlinear systems of equations using a fixed-point process.
In this context, the iterates are vectors generated from a fixed-point mapping function.
However, when considering the iterative solution of large-scale matrix equations, the iterates are low-rank matrices generated from a fixed-point process for which, generally, the mapping function changes in each iteration.
To enable acceleration of the iterative solution for these problems, we propose two novel generalizations of \ac{RRE}.
First, we show how to effectively compute \ac{RRE} for sequences of low-rank matrices.
Second, we derive a formulation of \ac{RRE} that is suitable for fixed-point processes for which the mapping function changes each iteration.
We demonstrate the potential of the methods on several numerical examples involving the iterative solution of large-scale algebraic Lyapunov and Riccati matrix equations.
}

\keywords{extrapolation methods, reduced rank extrapolation (RRE), fixed-point iteration, nonlinear equations, large-scale matrix equations, cycling mode}

\msc{65B05, % Extrapolation to the limit
65H05, % Numerical computation of solutions to single equations
15A24, % Matrix equations and identities
39B12 % Iteration theory, iterative and composite equations
}

\novelty{We propose novel extensions of \ac{RRE} for the iterative solution of large-scale sparse matrix equations: I) A scheme to efficiently implement \ac{RRE} for low-rank matrix iterates; and II) A scheme to apply \ac{RRE} for iterates generated by nonstationary fixed-point processes of the form~\eqref{eqn:xi_fp_iteration_dependent}.  % chktex 10
}

\maketitle

% chktex-file 3  % disable {} warnings for exponents etc.
% chktex-file 17 % disable warning about unmatched numbers of opening and
                 % closing brackets

\section*{Code and data availability}

The algorithms implemented in this paper and the data sets analyzed are available at:

\begin{center}
	\url{https://doi.org/10.5281/zenodo.14103908}
\end{center}

\section*{Acknowledgments}

This work was supported by ECSEL Joint Undertaking under grant agreement number 101007319 (AI-TWILIGHT).
We thank
Eda Oktay\orcidlink{0009-0000-8943-6908} and Fan Wang\orcidlink{0009-0000-8943-6908} %chktex 8
for their review of our codes and independent verification of the reproducibility of the numerical experiments.

\section{Introduction}
\Ac{RRE} was first introduced, independently, in~\cite{eddy1979extrapolating, kaniel1974least, mevsina1977convergence}.
It can be considered as a general technique to find the limit (or anti-limit) of a sequence of vectors
$\sequence{x_i}\subseteq\mathbb{R}^{d}$, and is most often applied in and analyzed based on the scenario that the iterates are generated by a fixed-point process $f$, i.e.,
\begin{equation}
	x_{i+1} = f(x_i)
	\label{eqn:xi_fp}
\end{equation}
with $i\in\NN$ and $x_1$ given.
The recent survey by Saad~\cite{saad2025survey} discussed various methods for accelerating fixed-point iterations on vectors, including RRE.
These scenarios arise most notably when considering the iterative solution of the nonlinear system of equations (which describe a fixed point of~\eqref{eqn:xi_fp})
\begin{equation}
	\mathscr F(x) = 0.
	\label{eqn:F(x)=0}
\end{equation}
However, the iterates can be generated by a process different from~\eqref{eqn:xi_fp}, for example:
\begin{enumerate}
	\item if the iterates are (possibly low-rank) matrices $\sequence{X_i}$ with $X_i \in \mathbb{R}^{d \times d}$ given in a factorized form, which shall be retained during the iteration; and
	\item if the iterates are generated by a fixed-point process that (although it still converges to a fixed-point $x_i \rightarrow x$) is now \emph{nonstationary}:
		\begin{equation}
			x_{i+1} = f_i(x_i),
			\label{eqn:xi_fp_iteration_dependent}
		\end{equation}
		with $x = f_i(x)$ for all $i = 1, 2, \dots$, and $x$ a fixed point.
\end{enumerate}
Two well-known examples in the first case are the Lyapunov equation and \ac{ARE}, where matrix iterates $\sequence{X_i}$ are associated with the iterative (matrix) solution to a matrix equation.
These equations play a prominent role in various problems in the field of
control theory, such as optimal control, model reduction and state estimation;
see, e.g.,~\cite[Section~3]{simoncini2016computational} and references therein.
The second scenario~\eqref{eqn:xi_fp_iteration_dependent} occurs naturally when we are iteratively solving a system of equations using a different preconditioner at every step, such as considered in~\cite{saad1993flexible}.
In that work, the author proposes an extension to \ac{GMRES} called \ac{FGMRES} which is more flexible in the sense that a different preconditioner at every step is supported.
(Similarly, later in Section~\ref{sct:flexible_rre}, we will extend \ac{RRE} to a flexible variant that supports nonstationary fixed-point processes of the form~\eqref{eqn:xi_fp_iteration_dependent}.)

There are practical applications in which both of the above scenarios occur simultaneously,
for example, when iteratively solving large-scale continuous-time Lyapunov equations using the \ac{ADI} method~\cite{penzl1999cyclic,li2002low,SchS24}. %chktex 36
The first formulation of the \ac{ADI} that iterates the residual alongside is
shown in~\cite{morBenKS13}, and~\cite{WolP16} shows its reinterpretation in the
Krylov subspace projection setting. Both point the way to a generalization for solving the \ac{ARE}s: the RADI~\cite{benner2018radi}, for which both scenarios above occur
%For the \ac{ADI} and RADI methods, the nonstationarity of the fixed-point process arises because the continuous-time equation is, at each step of the iterative process, transformed into a discrete-time problem.
%This transformation is parameterized by a shift parameter which, in general, is different for every iteration.
%As a result,
and the entire solution process can be viewed as an iterative process
of the form~\eqref{eqn:xi_fp_iteration_dependent} where, both, the iterates are
low-rank matrices \emph{and} the generating process changes in each iteration.

%Matrix-oriented versions of Krylov subspace methods, such as \ac{FGMRES}, have been used to solve large-scale Lyapunov and generalized (multiterm) Sylvester equations in, e.g.,~\cite{bollhofer2012structure,PalK2021,SimHao23}.
\Ac{RRE} has been applied in the context of large-scale matrix equations arising in transport theory in~\cite{el2013vector}.
The authors consider nonsymmetric \ac{ARE}s where the coefficients have a special structure
%that is found in transport problems such as radiation transfer.
%This structure allows the solution $X$ to be computed by the Hadamard product of a given matrix $T$ and the outer product of two vectors $u, v$:
%\begin{equation}
%	X = T \circ (u v^\TT).
%	\label{eqn:NARE_solution}
%\end{equation}
which allows the solution to be represented as $X = T \circ (u v^\TT)$ for a given matrix $T$ and two vectors $u, v$.
\chadded{Here $\circ$ refers to the element-wise product.}
%The authors subsequently apply \ac{RRE} to an iterative process on $u, v$.
Their numerical experiments show that \ac{RRE} significantly speeds up the standard iterative process, illustrating the potential of extrapolation methods in the context of matrix equations.

In this work, we derive efficient schemes (Algorithm~\ref{alg:lowrank_rre}) to implement \ac{RRE} for low-rank matrix iterates and to apply \ac{RRE} for iterates generated by nonstationary fixed-point processes of the form~\eqref{eqn:xi_fp_iteration_dependent} (Algorithm~\ref{alg:nonstationary_rre}). Combining both extensions, we also demonstrate the acceleration potential of RRE for general large-scale sparse \ac{ARE}s for which the solution admits an effective low-rank approximation (Algorithm~\ref{alg:nonstationary_lowrank_rre}).

The rest of the paper is organized as follows.
In Section~\ref{sct:preliminary}, we briefly introduce the formulation of \ac{RRE} that is commonly used in literature.
The generalization of this formulation is presented in
Section~\ref{sct:method}, with its application to solving large-scale matrix
equations subsequently discussed in Section~\ref{sct:application}.
In Section~\ref{sct:numerical_examples}, we demonstrate the potential of the method on several numerical examples of large-scale matrix equations.
Finally, we provide conclusions and an outlook in Section~\ref{sct:conclusion}.

%\subsection{Notation}
Throughout the paper, we use $\cdot^\TT$ to denote real transposition.
$\norm{\cdot}$ may denote the spectral or Frobenius norm if the argument is a matrix,
and the Euclidean norm if the argument is a vector.
$\abs{\cdot}$ denotes the absolute value of a scalar.
The vector of all ones is denoted as $\ones$.
$\NN$ is the set of positive integers,
$\RRneg$ is the negative real axis (excluding the origin),
and $\mathbb{R}^{m\times n}$ is the set of real-valued
$m$-by-$n$ matrices.
$\machineprecision$ denotes the machine precision in IEEE double precision floating-point arithmetic.
We use $\blkdiag(A_1,A_2,\dots,A_k)$ to denote the block diagonal matrix formed by aligning the block $A_1,A_2,\dots,A_k$ along the main diagonal.

%\subsection{Structure}

\section{Preliminaries on RRE}\label{sct:preliminary}

In this section, we give a brief derivation of the method in the context of solving the linear system
\begin{equation}
	A x = b.
	\label{eqn:Ax=b}
\end{equation}
The extension of \ac{RRE} to the more general case of iteratively solving the nonlinear system~\eqref{eqn:F(x)=0} is simple in the sense that the method and algorithm stay the same (although convergence guarantees are no longer easy to derive).
We assume that $A$ is regular,
i.e., system~\eqref{eqn:Ax=b} permits a unique solution.

In many applications, the matrix $A$ is too large to be explicitly stored
or for its inverse $A^{-1}$ to be computed.
Instead, only the result of $Ay$
for a given vector $y$ may be computed.
In that case, one could attempt to find the solution $x$
by iterating the general linear fixed-point procedure
\begin{equation}
\begin{aligned}
	x_{i+1} = f(x_i)
	:={}& M^{-1} (N x_i + b) \\
	={}& x_i - M^{-1} (Ax_i - b)
\end{aligned}
\label{eqn:xi+1=Axi+b}
\end{equation}
with $A = M - N$ and an initial vector $x_1$.
The final formulation is often referred to as the second normal form of a consistent linear iteration~\cite[Section~2.2.2]{Hac16}.
There are two practical issues associated with the fixed-point iteration~\eqref{eqn:xi+1=Axi+b}.
First, the iteration~\eqref{eqn:xi+1=Axi+b} converges if and only if the eigenvalues of $M^{-1}N$ are on the open unit disk~\cite[Section~11.2.3]{golub2013matrix}.
Second, even if the spectral radius of $M^{-1}N$ is strictly less than 1, the
convergence of the iteration~\eqref{eqn:xi+1=Axi+b} can be slow when the radius
is close to 1.
The method of \ac{RRE} can mitigate both these issues, in the sense that it only requires most eigenvalues of $M^{-1}N$ have modulus sufficiently small (we will make this statement more precise in the sequel).

At this point,
there are two notions of a residual:
the residual of the underlying equation~\eqref{eqn:Ax=b},
\begin{align}
	\reseq(y) &:= \mathscr F(y) = b - Ay
	,
	\label{eqn:residual:equation}
\intertext{%
and the residual of the fixed-point iteration~\eqref{eqn:xi+1=Axi+b},
}
	\resit(y) &:= f(y) - y = M^{-1} \reseq(y)
	.
	\label{eqn:residual:iteration}
\end{align}
Clearly, we have that \(\reseq(x) = \resit(x) = 0\) for the solution~$x$ of~\eqref{eqn:Ax=b}.
Most derivations of \ac{RRE} are only concerned with the case $M=I$,
for which both these notions of the residual coincide,
e.g.,~\cite{% underlying equation
	mevsina1977convergence, % X = AX + f (1)
	eddy1979extrapolating,  % x = Ax + b, but never spelled out
	sidi1988extrapolation,  % x = Ax + b (2.8)
	sidi1991efficient,      % x = Ax + b (3.1)
	sidi1998upper}          % Ax = b (1.1); defines eqn:residual:equation
\unskip,
or with the residual of the iteration~\eqref{eqn:residual:iteration},
e.g.,~\cite{% underlying equation
	kaniel1974least,        % Ax = f, just before (1)
	sidi2020convergence}    % x = f(x) (1.1)
\unskip.
For the moment, we only consider the case that $\reseq(\cdot) = \resit(\cdot)$ everywhere.
\unskip\footnote{We will lift this restriction in Section~\ref{sct:flexible_rre}.}

The main idea of all extrapolation methods in general is to construct an extrapolant $\widehat{x}$ that is a linear combination of $n$ (which is fixed) iterates $\{ x_i \}_{i=1}^{n}$, i.e.,
\begin{equation}
	\widehat{x} = \sum_{i=1}^{n} \gamma_i x_i,
	\label{eqn:rre_extrapolant}
\end{equation}
for some weights $\gamma_i$, $i = 1,\ldots, n$.
The difference among different extrapolation methods (such as \ac{RRE}) lies in how these weights are determined.
For \ac{RRE}, the weights are chosen to minimize the
residual~\eqref{eqn:residual:iteration} of a convex combination of iterates,
\begin{equation}
	\gamma = \argminst{%
		g\in\mathbb R^n
	}{%
		\norm{ \resit \left( \sum_{i=1}^{n} g_i x_i \right) }
	}{%
		\sum_{i=1}^{n} g_i = 1.
	}
	\label{eqn:residual_extrapolant}
\end{equation}
The convexity constraint imposed on the weights $\gamma_i$,
\begin{equation}
	\sum_{i=1}^{n} \gamma_i = 1,
	\label{eqn:gamma_sum_to_1}
\end{equation}
is an assumption that is shared by almost all extrapolation methods (not just vector extrapolation) using linear combinations of iterates~\cite[Section~0.3]{sidi2003practical}.
The consequence of this assumption is that the extrapolation method becomes shift-independent.
That is, shifting all iterates by some constant term $\bar{x}$ also shifts the extrapolant by $\bar{x}$.
Furthermore,
using~\eqref{eqn:gamma_sum_to_1},
it is easy to verify that
\begin{equation}
	\reseq\left(\sum_{i=1}^n \gamma_i x_i\right) = \sum_{i=1}^n \gamma_i \reseq(x_i)
	\qquad\text{and}\qquad
	\resit\left(\sum_{i=1}^n \gamma_i x_i\right) = \sum_{i=1}^n \gamma_i \resit(x_i)
\end{equation}
hold; irrespective of \(M\).
By the observation that $\resit(x_i) = x_{i+1} - x_i$,
we can thus write the extrapolant's residual as
\begin{equation}
	\resit\left( \sum_{i=1}^{n} \gamma_i x_i \right)
	= \sum_{i=1}^{n} \gamma_i (x_{i+1} - x_i).
\end{equation}
Note that in the final term of the sum ($i=n$),
we need the additional term $x_{n+1}$ to evaluate $\resit(x_i)$.
This final term is used commonly in literature to express the basic optimization step involved in \ac{RRE}~\cite{el2013vector, kaniel1974least, sidi2020convergence}:
\begin{equation}
	\gamma = \argminst{%
		g\in\mathbb R^n
	}{%
		\norm{ \sum_{i=1}^{n} g_i \left( x_{i+1} - x_i \right) }
	}{%
		\sum_{i=1}^{n} g_i = 1.
	}
	\label{eqn:rre_opt}
\end{equation}
The optimization~\eqref{eqn:rre_opt} is convex and can be solved efficiently using a QR-decomposition as outlined in~\cite{sidi1991efficient}.
Once solved, the coefficients $\gamma$ are used to obtain the extrapolant~\eqref{eqn:rre_extrapolant}.
We summarize this procedure in Algorithm~\ref{alg:rre}.

\begin{algorithm}
	\caption{%
		\ac{RRE} for stationary processes~\eqref{eqn:xi_fp}
		in standard formulation~\eqref{eqn:rre_opt}.
	}%
	\label{alg:rre}
	\KwIn{%
		Iterates $x_1, \ldots, x_n, x_{n+1}$.
	}
	\KwOut{%
		Extrapolant~$\widehat x$.
	}
	\medskip
	\Fn{$\RRE_\Delta(x_1, \ldots, x_n, x_{n+1})$}{%
		Initialize matrix of differences $U \gets [
			x_2 - x_1 | \dots | x_{n+1} - x_n
		] \in\mathbb R^{d\times n}$\;
		$\gamma\gets\RRE(U)$\;
		$\widehat x \gets \sum_{i=1}^n \gamma_i x_i$\;
	}
	\medskip
	\Fn{$\RRE(U)$}{%
		Compute $R\in\mathbb R^{n\times n}$ of economy-size QR-decomposition $U = QR$\;
		Solve $R^\TT R \alpha = \ones$ for $\alpha\in\mathbb R^n$\;
		$\gamma\gets\alpha / \sum_{i=1}^n \alpha_i$\;
		\Return{$\gamma$}\;
	}
\end{algorithm}

\begin{algorithm}
	\caption{Non-cycling \ac{RRE} mode.}\label{alg:noncycling_rre}
	\KwIn{Initialization $x_1$, window size $n \in \NN$.}
	\KwOut{Extrapolated sequence (of extrapolated iterates only) $\widehat{x}_{n}, \widehat{x}_{n+1}, \ldots$}
	\For{$i\gets 1, 2, \ldots$}{%
		$x_{i+1} \gets f(x_i)$\;
		\If{$i\geq n$}{%
			$\widehat{x}_i \gets \RRE_\Delta(x_{i-n+1}, \ldots, x_i, x_{i+1})$\;
			\lIf{$\widehat{x}_i$ converged}{break}
		}
	}
\end{algorithm}

\begin{algorithm}
	\caption{Cycling \ac{RRE} mode.}\label{alg:cycling_rre}
	\KwIn{%
		Initialization $x_1$,
		window size $n \in \NN$.
	}
	\KwOut{%
		Extrapolated sequence $x_1, x_2, \ldots$
	}
	% Ensure this matches run_simple.m:
	\For{$i\gets 1, 2, \ldots$}{%
		$x_{i+1} \gets f(x_i)$\;
		\lIf{should start new cycle}{$x_{i+1} \gets \RRE_\Delta(x_{i-n+1}, \ldots, x_i, x_{i+1})$}%
		\label{alg:cycling_rre:new_cycle}
		\lIf{$x_{i+1}$ converged}{break}
	}
\end{algorithm}

\subsection{Application modes}

In extrapolation methods, typically the extrapolation step is applied multiple times according to the application mode.
We will here introduce two such application modes.

The first application mode we discuss is the \emph{non-cycling} \ac{RRE} mode (Algorithm~\ref{alg:noncycling_rre}), which is also referred to as $n$-Mode in~\cite{sidi2020convergence}.
In this mode, we compute (for fixed window size $n$) a sequence $\widehat{x}_n, \widehat{x}_{n+1}, \dots$,
until convergence is achieved.
This mode can be used in a non-intrusive context in the sense that it does not require knowledge of the underlying process generating these iterates.
This is a necessity, for example, in the case of sequences being generated using commercial software.
For the scalar case, this mode is similar to Aitken's $\Delta^2$ method~\cite{aitken1927bernoulli} in the sense that a window (of fixed size) is shifted over the sequence of iterates to produce another sequence of iterates.

For some problems, the non-cycling scheme only provides a modest speed up in terms of convergence rate.
To improve this, a second application mode can be used: \emph{cycling} \ac{RRE} mode (Algorithm~\ref{alg:cycling_rre}).
In each cycle, the fixed-point iteration~\eqref{eqn:xi_fp} is initialized using the extrapolant of the previous cycle.
Since the extrapolant is used to restart the fixed-point process, the mapping~$f$ must be known.
A new cycle may start, for example, if $n$ divides $i$ (line~\ref{alg:cycling_rre:new_cycle}).
Some convergence properties of cycling \ac{RRE} mode for the iterative solutions of linear and nonlinear equations are studied in~\cite{Sidi1992upper, sidi2020convergence}.

\subsection{Connections to other iterative methods}

Finally, we remark that in the case of solving linear systems using iterative methods, there exist strong connections between these application modes of vector extrapolation and other iterative methods.
In the context of Krylov subspace methods, it was shown in~\cite{sidi1988extrapolation} that cycling, respectively non-cycling, \ac{RRE} is equivalent to the restarted, respectively non-restarted, method of \ac{GMR}.
Several (mathematically equivalent) implementations of \ac{GMR} have been proposed in literature such as in~\cite{axelsson1980conjugate, eisenstat1983variational, young1980generalized} and the commonly used \ac{GMRES}~\cite{saad1986gmres}.
\Ac{RRE} is also equivalent to the method of \ac{DIIS} which was first proposed in~\cite{pulay1980convergence} to accelerate the iterative solutions of linear systems of equations.
An improved version focused on the \ac{SCF} method is given in~\cite{pulay1982improved}.
The numerical properties of several implementations of \ac{DIIS} are compared in~\cite{shepard2007some}.

\section{Methodological extensions of RRE}\label{sct:method}

In this section, we discuss two extensions to \ac{RRE} that make it suitable for iterative solvers for large-scale matrix equations: (i) In Section~\ref{sct:matrix_rre}, we treat the case that the iterates consist of low-rank matrices; and (ii) In Section~\ref{sct:flexible_rre}, we treat the case that the iterates are generated from a fixed-point iteration with nonstationary mapping~\eqref{eqn:xi_fp_iteration_dependent}. % chktex 10

These two extensions are generic in nature and can be applied independently.
Nevertheless, in Section~\ref{sct:application} we discuss in more detail the application of RRE with both extensions to the important case of iteratively solving large-scale matrix equations.

\subsection{RRE for low-rank matrix sequences}\label{sct:matrix_rre}
Similar to how \ac{RRE} is used to accelerate the iterative solutions of systems with vector solutions, we are interested in accelerating the iterative solution of the matrix-valued system
\begin{equation}
	\mathscr F(X) = 0.
	\label{eqn:F(X)=0}
\end{equation}
We assume that a solution $X \in \mathbb{R}^{d \times d}$ exists and is symmetric.
The symmetry assumption is for convenience of notation only: the method can be extended to the non-symmetric case as will be briefly discussed at the end of this section.
The function $\mathscr F: \mathbb{R}^{d \times d} \rightarrow \mathbb{R}^{d \times d}$ in equation~\eqref{eqn:F(X)=0} may be nonlinear.
%We assume that a direct solution for the system~\eqref{eqn:F(X)=0} cannot be obtained (for example, because the computational load is too high).
We consider an iterative solution process for~\eqref{eqn:F(X)=0} by the fixed-point iteration:
\begin{equation}
	X_{i+1} = F(X_i)
	\label{eqn:X=F(X) fp}
\end{equation}
with initialization $X_1$ and $F: \mathbb{R}^{d \times d} \rightarrow \mathbb{R}^{d \times d}$.
Note that convergence of the iteration~\eqref{eqn:X=F(X) fp} to~$X$ depends on~$F$ and the initialization.
For example, for the matrix equations we consider in this article, convergence is guaranteed under mild technical assumptions (which will be discussed further in Section~\ref{sct:application}).

One immediate possibility is to apply \ac{RRE} to the vectorized sequence of matrices $\sequence{\vec(X_i)}$,
where $\vec(\cdot)$ stacks the columns of its argument into a vector.
However, this is not possible for large-scale problems where the iterates are too large (i.e., $d$ is too large) to fit in memory.
Instead, iterative solvers of large-scale matrix equations, such as RADI, exploit the fact that, for many practical problems, the desired solution can be well-approximated by a low-rank matrix~\cite{benner2016solution}.
Such solvers produce low-rank iterates in factorized form:
\begin{equation}
	X_i = Z_i D_i Z_i^\TT,
	\label{eqn:Xi_lr_decomp}
\end{equation}
with $Z_i \in \mathbb{R}^{d \times z_i}$, $D_i = D_i^\TT \in \mathbb{R}^{z_i \times z_i}$ and $z_i \ll d$.
Algorithmically, the fixed-point iteration~\eqref{eqn:X=F(X) fp} is implemented by only operating on the low-rank factors (for example, \ac{ADI} and RADI).
That is, conceptually,
\begin{equation}
	(Z_{i+1}, D_{i+1}) = F(Z_i, D_i)
	,
	\label{eqn:X=F(X) lr}
\end{equation}
where potentially $z_{i+1} \neq z_i$.

Substituting the low-rank decomposition~\eqref{eqn:Xi_lr_decomp} into the formulation of vector \ac{RRE}~\eqref{eqn:rre_opt} gives
\begin{equation}
	\gamma = \argminst{%
		g\in\mathbb R^n
	}{%
		\norm{ \sum_{i=1}^{n} g_i \left( Z_{i+1} D_{i+1} Z_{i+1}^\TT - Z_i D_i Z_i^\TT \right) }
	}{%
		\sum_{i=1}^{n} g_i = 1,
	}
	\label{eqn:matrre_opt}
\end{equation}
where $\norm{ \cdot }$ denotes the corresponding matrix norm.

Solving this optimization problem, again, gives us a set of weights $\gamma$ for which a low-rank decomposition of the extrapolant can be constructed:
\begin{equation}
\begin{aligned}
	\widehat{X} & := \sum_{i=1}^{n} \gamma_i Z_i D_i Z_i^\TT \\
	& = \begin{bmatrix} Z_1 & \dots & Z_n \end{bmatrix}
	\begin{bmatrix} \gamma_1 D_1 \\ & \ddots \\ & & \gamma_n D_n \end{bmatrix}
	\begin{bmatrix} Z_1 & \dots & Z_n \end{bmatrix}^\TT
	=: \widehat{Z} \widehat{D} \widehat{Z}^\TT.
\end{aligned}
\label{eqn:matrre_extrapolant}
\end{equation}
The rank of the extrapolant is at most $\sum_{i=1}^{n} z_i$.
Depending on the problem, the rank can be significantly lower than this because the columns of $\widehat{Z}$ may be linearly dependent.
For example, in \ac{ADI} and RADI, we have that $\range(Z_i) \subseteq \range(Z_{i+1})$ (see~\cite[eq.~(12)]{benner2018radi}).
Hence, for these methods the rank of the extrapolant is at most $z_n$.

From a computational point of view, solving the optimization problem~\eqref{eqn:matrre_opt} is usually not feasible because the problem size $d$ is so large that we are unable to assemble $X_1, \ldots, X_{n+1} \in\mathbb R^{d \times d}$.
However, if we restrict $\norm{\cdot}$ to be the Frobenius norm or spectral norm, then we can exploit the following unitary invariance of \emph{rectangular matrices} in developing a computationally scalable solution strategy.
We should note that the result has been used without a proof in~\cite[equation~(4.7)]{penzl1999cyclic}.

\begin{lem}\label{lem:norm}
	For a given \chadded{symmetric} matrix $\Delta \in\mathbb{R}^{d \times d}$,
	let $Q \in \mathbb{R}^{d \times q}$ with $q \leq d$ be a matrix such that $Q^\TT Q = I_q$ and $\range(\Delta) \subseteq \range(Q)$.
	\chadded{Then}
\begin{equation}
	\norm{ \Delta } = \norm{ Q^\TT \Delta Q },
	\label{eqn:thm_norm_equivalence}
\end{equation}
\chadded{where $\norm{ \cdot }$ represents either the Frobenius norm or spectral norm.}
\end{lem}
\begin{proof}
The equality~\eqref{eqn:thm_norm_equivalence} is similar to the classical result that the Frobenius and spectral norms are invariant under orthogonal transformations~\cite[Section~IV.2]{bhatia2013matrix}, \cite[Section~7.4.7]{horn2012matrix}, \cite[Section~6.2]{higham2002accuracy}, where $Q$ is orthogonal (and hence square). % chktex 2

\chadded{%
Here, the subtlety is that $Q$ is allowed to be non-square in case of rank deficiency of $\Delta$.
Since $\range(\Delta) \subseteq \range(Q)$,
there exists some $M\in\mathbb R^{q\times d}$ with $\Delta = QM$,
such that $QQ^\TT \Delta = \Delta$, and by the symmetry of $\Delta$, $\Delta QQ^\TT = \Delta$.
Now choose $Q_\perp \in \mathbb{R}^{d\times(d-q)}$ such that its columns form an orthonormal basis for the orthogonal complement of $\range(Q)$, that is,
$\widetilde Q := \begin{bmatrix}
		Q & Q_\perp
	\end{bmatrix} \in\mathbb R^{d\times d}$
is an orthogonal matrix.
Then, since $\Delta Q_\perp = M^\TT Q^\TT Q_\perp=0$ and similarly
$Q_\perp^\TT \Delta = 0$,}
\begin{equation*}
	\norm{\Delta}
	=
	\left\| \widetilde Q^\TT \Delta \widetilde Q \right\|
	=
	\norm{
		\begin{bmatrix}
			Q^\TT \Delta Q & 0 \\
			0 & 0
	\end{bmatrix}}
	=
	\norm{Q^\TT \Delta Q},
\end{equation*}
\chadded{where we have also used the orthogonal invariance of both the Frobenius norm and the spectral norm.}
\end{proof}

% The sum $z$ has been used before, but we don't need this symbol often,
% and the previous use is quite far away.
We now continue to derive a computationally scalable formulation of \ac{RRE} for low-rank matrix sequences~\eqref{eqn:Xi_lr_decomp}.
Define $z := \sum_{i=1}^{n+1} z_i$ and let
$Q  \in\mathbb{R}^{d\times z}$ and
$\tilde Z_i\in\mathbb{R}^{z\times z_i}$ be given by the thin QR-decomposition
\begin{equation}
	\begin{bmatrix}
		Z_1 & \dots & Z_{n+1}
	\end{bmatrix}
	= Q \begin{bmatrix}
		\tilde Z_1 & \dots & \tilde Z_{n+1}
	\end{bmatrix}
	\in\mathbb{R}^{d\times z}
	.
\end{equation}
Observe that
\begin{align}
	\Delta
	:={}& \sum_{i=1}^{n} g_i \left(
		Z_{i+1} D_{i+1} Z_{i+1}^\TT - Z_i D_i Z_i^\TT
	\right) \\
\intertext{can be written as}
	={}& \begin{bmatrix}
		Z_1 & \dots & Z_{n+1}
	\end{bmatrix}
	\begin{bmatrix}
		-g_1 D_1 \\
		& (g_1 - g_2) D_2 \\
		&& \ddots \\
		&&& (g_{n-1} - g_n) D_n \\
		&&&& g_n D_{n+1}
	\end{bmatrix}
	\begin{bmatrix}
		Z_1^\TT \\
		\vdots \\
		Z_{n+1}^\TT
	\end{bmatrix}
\end{align}
and that $\range(\Delta) \subseteq \range(Q)$.
Following Lemma~\ref{lem:norm},
$\norm{\Delta} = \norm{Q^\TT \Delta Q}$ holds,
such that the low-rank \ac{RRE}~\eqref{eqn:matrre_opt} is equivalent to
\begin{equation}
	\gamma = \argminst{%
		g\in\mathbb R^n
	}{%
		\norm{ \sum_{i=1}^{n} g_i \left( \tilde Z_{i+1} D_{i+1} \tilde Z_{i+1}^\TT - \tilde Z_i D_i \tilde Z_i^\TT \right) }
	}{%
		\sum_{i=1}^{n} g_i = 1
		,
	}
	\label{eqn:matrre_opt_efficient}
\end{equation}
where all the matrix products involved are much smaller than in formulation~\eqref{eqn:matrre_opt},
$\mathbb{R}^{z\times z}$ as opposed to $\mathbb{R}^{d\times d}$,
assuming $n \ll d$ such that also $z := \sum_{i=1}^{n+1} z_i \ll d$.
At this point,
it becomes viable to apply $\vec(\cdot)$ to the matrices $\tilde Z_i D_i \tilde Z_i^\TT$, $i = 1, \ldots, n+1$
and use Algorithm~\ref{alg:rre}.
(Note that symmetry of these matrices can be exploited in the implementation of $\vec(\cdot)$.)

A stronger condition than $\range(Z_i)\subseteq\range(Z_{i+1})$ may enable us to further reduce the computational complexity.
For example, the iterates produced by \ac{ADI} and RADI reveal increments
\unskip\footnote{%
	The curious reader may skip ahead to Algorithm~\ref{alg:radi} on page~\pageref{alg:radi}.
}
$V_{i+1}\in\mathbb{R}^{d\times v_{i+1}}$ and
$\tilde{D}_{i+1}\in\mathbb{R}^{v_{i+1}\times v_{i+1}}$ such that
$v_{i+1} = z_{i+1}-z_i \in\NN$ and
\begin{equation}
	Z_{i+1} = \begin{bmatrix}
		Z_i & V_{i+1}
	\end{bmatrix}
	\qquad\text{and}\qquad
	D_{i+1} = \begin{bmatrix}
		D_i \\
		& \tilde D_{i+1}
	\end{bmatrix}
	.
	\label{eqn:Xi_lr_decomp_inc}
\end{equation}
The optimization problem~\eqref{eqn:matrre_opt} then reads
\begin{equation}
	\gamma = \argminst{%
		g\in\mathbb R^n
	}{%
		\norm{ \sum_{i=1}^{n} g_i V_{i+1} \tilde D_{i+1} V_{i+1}^\TT }
	}{%
		\sum_{i=1}^{n} g_i = 1
		.
	}
	\label{eqn:matrre_opt_more_efficient:interim}
\end{equation}
Observe that
\begin{equation}
	\Delta
	= \sum_{i=1}^{n} g_i V_{i+1} \tilde D_{i+1} V_{i+1}^\TT
	= \begin{bmatrix}
		V_2 & \dots & V_{n+1}
	\end{bmatrix}
	\begin{bmatrix}
		g_1 \tilde D_2 \\
		& \ddots \\
		&& g_n \tilde D_{n+1}
	\end{bmatrix}
	\begin{bmatrix}
		V_2^\TT \\
		\vdots \\
		V_{n+1}^\TT
	\end{bmatrix}
	,
\end{equation}
whose rank is limited by $q := z_{n+1} - z_1 < z_{n+1} < z$.
Let $\tilde Q\in\mathbb{R}^{d\times q}$ and $\tilde V_i\in\mathbb{R}^{q\times v_i}$ be given by the thin QR-decomposition
\begin{equation}
	\begin{bmatrix}
		V_2 & \dots & V_{n+1}
	\end{bmatrix} =
	\tilde Q \begin{bmatrix}
		\tilde V_2 & \dots & \tilde V_{n+1}
	\end{bmatrix}
	\in\mathbb{R}^{d\times q}
	.
\end{equation}
Obviously, $\range(\Delta)\subseteq\range(\tilde Q)$,
and Lemma~\ref{lem:norm} implies that
\begin{equation}
	\gamma = \argminst{%
		g\in\mathbb R^n
	}{%
		\norm{ \sum_{i=1}^{n} g_i \tilde V_{i+1} \tilde D_{i+1} \tilde V_{i+1}^\TT }
	}{%
		\sum_{i=1}^{n} g_i = 1
		,
	}
	\label{eqn:matrre_opt_more_efficient}
\end{equation}
is equivalent to formulation~\eqref{eqn:matrre_opt_efficient}.
Here, all matrix products can be evaluated in $\mathbb{R}^{q\times q}$, and $q<z$.

\begin{figure}[t]
	\centering
	\resizebox{0.9\linewidth}{!}{% chktex-file 1  % nodes without options are whitespace terminated.
                 % we could alternatively define \node as a safe command
% chktex-file 11 % foreach syntax uses the dots to describe sequences
% chktex-file 26 % singleton semicolon trigger this
% chktex-file 36 % tikz syntax needs brackets with non-whitespace char in front

% Whether to show residual blocks:
\providecommand\ifresiduals\iffalse

% Configure height, width, and separation between colorful blobs:
\setlength\myH{5em}%
\setlength\myW{2ex}%
\setlength\myS{1ex}%

% Low-rank factorization:
\newcommand\nodeVDV[3]{%
	\node (#1) {#2};
	\draw[fill, color=#3!50]
		(#1.center)
		++(-1.5\myW-\myS-0.5\myH,-0.5\myH) % center node in white space
		rectangle ++(\myW,\myH)
		++(\myS,-\myW)
		rectangle ++(\myW,\myW)
		++(\myS,-\myW)
		rectangle ++(\myH,\myW)
		;
	\coordinate (#1 north) at (current bounding box.north);
	\coordinate (#1 south) at (current bounding box.south);
}%

% Residual factor:
\newcommand\nodeR[3][1]{%
	\ifresiduals
	\node (#2) {#3};
	\draw[fill, gray!50]
		(#2.west) ++(-\myS,-0.5\myH)
		foreach \distance in {1,...,#1} {
			rectangle ++(-\myW,\myH)
			++(-\myS,-\myH)
		};
	\else
	\coordinate[yshift=.5ex] (#2);
	\fi
}%

\begin{tikzpicture}[
	thick,
	X/.style = {
		fill = gray!20,
		minimum size = \myH,
	},
	heading/.style = {
		above = \myS,
		inner sep = 0pt,
	},
	every node/.style = {
		anchor = base,
	},
	every outer matrix/.style = {
		inner sep = 0pt,
	},
]
	\matrix (foo) [
		row sep = 2\myS,
		column sep = \myS,
	] {
		% Row 1:
		\node[X] (X1) {$X_1$}; &
		\node {=}; &
		\nodeVDV{V11}{$\hphantom{\tau_1} Z_1 D_1 Z_1^\TT$}{mycolor1} &&&&& [2\myS]
		\nodeR{R1}{$R_1$} & [4\myS]
		\coordinate (R1+); \\

		% Row 2:
		\node[X] (X2) {$X_2$}; &
		\node {=}; &
		\nodeVDV{V21}{$\hphantom{\tau_1} Z_1 D_1 Z_1^\TT$}{mycolor1} &
		\node {+}; &
		\nodeVDV{V22}{$\hphantom{\tau_2} V_2 \tilde D_2 V_2^\TT$}{mycolor2} &&&
		\nodeR{R2}{$R_2$} &
		\coordinate (R2+); & [4\myS]
		\node[draw, fill=gray!20] (RRE) {RRE}; \\

		% Row 3:
		\node[X] (X3) {$X_3$}; &
		\node {=}; &
		\nodeVDV{V31}{$\hphantom{\tau_1} Z_1 D_1 Z_1^\TT$}{mycolor1} &
		\node {+}; &
		\nodeVDV{V32}{$\hphantom{\tau_2} V_2 \tilde D_2 V_2^\TT$}{mycolor2} &
		\node {+}; &
		\nodeVDV{V33}{$\hphantom{\tau_3} V_3 \tilde D_3 V_3^\TT$}{mycolor3} &
		\nodeR{R3}{$R_3$} &
		\coordinate (R3+); \\

		% Row 4:
		\node[X, label={[heading] Extrapolant}] {$\widehat X_{\ifresiduals 3\else 2\fi}$}; &
		\node{=}; &
		\nodeVDV{V1}{$\tau_1 Z_1 D_1 Z_1^\TT$}{mycolor1} &
		\node {+}; &
		\nodeVDV{V2}{$\tau_2 V_2 \tilde D_2 V_2^\TT$}{mycolor2} &
		\ifresiduals
		\node {+}; &
		\nodeVDV{V3}{$\tau_3 V_3 \tilde D_3 V_3^\TT$}{mycolor3} &
		\else
		&&
		\fi
		\coordinate[yshift=.5ex] (RE); \\
	};

	% Draw column headers:
	\node[heading] at (V11 north -| X1) {Iterate};
	\node[heading] at (V11 north -| V2) {Low-rank factors};
	\ifresiduals
	\node[heading] at (V11 north -| R1) {Residual factors};
	\fi

	% Draw dashed lines in between rows:
	\foreach \row in {1, 2, 3} {
		\draw[
			dashed,
			transform canvas = { yshift = -\myS },
			shorten > = -\myS,
		] (X\row.west |- V\row1 south) ++(-\myS,0) -- (V\row1 south -| R\row.east);
	}

	% Draw RRE apparatus:
	\coordinate (junction) at (R1 -| R1+); % baseline adjustment
	\draw[->] (R1) -| (RRE);
	\draw (R2) -- (R2 -| junction);
	\draw (R3) -| (junction);
	\draw[->] (RRE) |- (RE);
	\node[anchor=south east] at (RE -| RRE) {$\tau_1, \tau_2\ifresiduals, \tau_3\fi$};
\end{tikzpicture}%
}
	\caption{%
		Schema illustrating the application of \ac{RRE} to the first 3 low-rank
		iterates $X_1, X_2, X_3$,
		where the weights $\tau_1, \tau_2$ are as in~\eqref{eqn:matrre_tau}.
	}%
	\label{fig:lowrank_rre}
\end{figure}
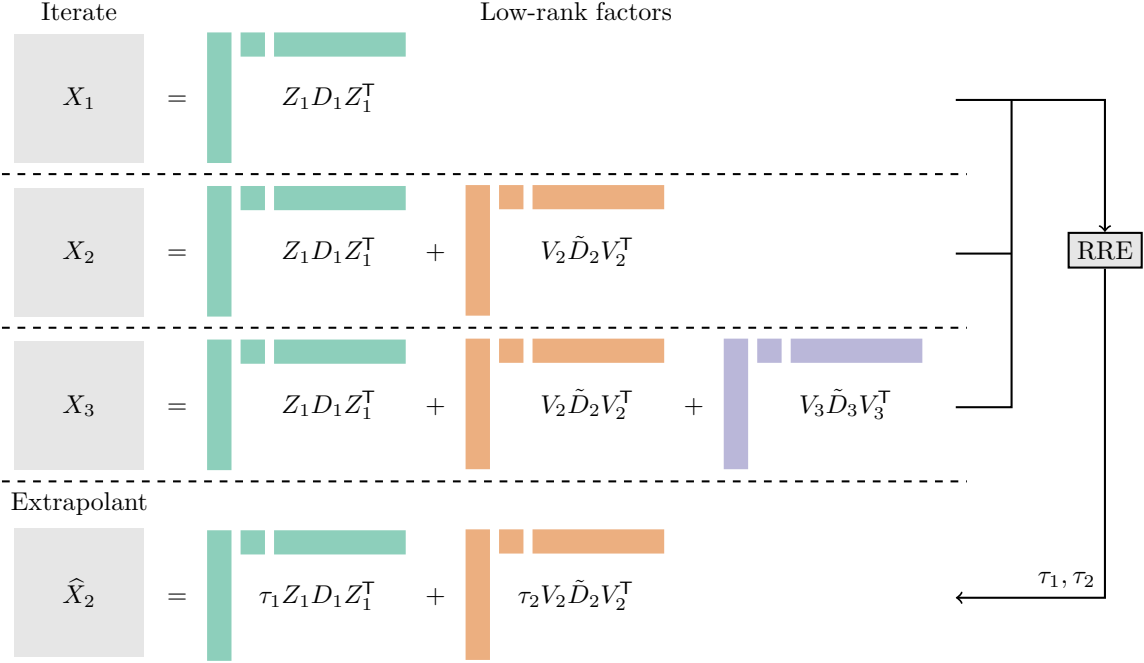

Figure~\ref{fig:lowrank_rre} gives a graphical overview on the procedure.
In this figure, we also highlight the fact that the iterates are constructed by low-rank updates~\eqref{eqn:Xi_lr_decomp_inc}.
Thus, the extrapolant $\widehat{X}$ can be written as a weighted sum of these low-rank updates:
\begin{equation}
	\widehat{X}
	= \sum_{i=1}^{n} \gamma_i X_i
	= \tau_1 Z_1 D_1 Z_1^\TT + \sum_{i=2}^{n} \tau_i V_i \tilde D_i V_i^\TT,
	\label{eqn:matrre_tau}
\end{equation}
with $\tau_i = \sum_{j=i}^{n} \gamma_j$, $i = 1, \ldots, n$.

\begin{remark}
	The extrapolant~$\widehat X$ is written without a subscript
	to emphasize that the extrapolation procedure is oblivious to offsets in the comprising sequence of iterates;
	compare Algorithm~\ref{alg:rre}.
	In Figure~\ref{fig:lowrank_rre}, however,
	foreshadowing Figure~\ref{fig:matrre_radi},
	we assign the extrapolant to~$\widehat X_2$ with a subscript
	to hint at the fact that the extrapolant~$\widehat X_2$ is essentially based on~$X_1$ and~$X_2$
	while~$X_3$ is a mere algorithmic necessity;
	compare Algorithms~\ref{alg:noncycling_rre} and~\ref{alg:cycling_rre}.
	In Section~\ref{sct:application},
	we will demonstrate how an extrapolant~$\widehat X_3$ can instead be obtained without computing~$X_4$.
\end{remark}

\begin{algorithm}
	\caption{%
		\ac{RRE} for low-rank matrix sequences~\eqref{eqn:matrre_opt_more_efficient}.
	}%
	\label{alg:lowrank_rre}
	\KwIn{%
		Iterates $X_1, \ldots, X_{n+1}$ given by
		$Z_1\in\mathbb{R}^{d\times z_1}$
		and
		$D_1\in\mathbb{R}^{z_1\times z_1}$
		as well as
		$V_i\in\mathbb{R}^{d\times v_i}$
		and
		$\tilde D_i\in\mathbb{R}^{v_i\times v_i}$
		fulfilling~\eqref{eqn:Xi_lr_decomp} and~\eqref{eqn:Xi_lr_decomp_inc}
		for $i=2,\ldots,n+1$.
	}
	\KwOut{%
		Extrapolant $\widehat{X} = \widehat{Z} \widehat{D} \widehat{Z}^\TT$.
	}
	\medskip
	\Fn{$\RRE_{\Delta}^{\LR}(X_1, \ldots, X_n, X_{n+1})$}{
		Compute triangular factor of thin QR-decomposition:
		\begin{equation*}
			\begin{bmatrix}
				V_{2} & \dots & V_{n+1}
			\end{bmatrix}
			= \tilde Q \begin{bmatrix}
				\tilde V_{2} & \dots & \tilde V_{n+1}
			\end{bmatrix}
		\end{equation*}
		\vspace{-\baselineskip}\;%
		Initialize matrix of differences:
		\begin{equation*}
			U \gets\begin{bmatrix}
				\vec(\tilde V_2 \tilde D_2 \tilde V_2^\TT) & \dots &
				\vec(\tilde V_{n+1} \tilde D_{n+1} \tilde V_{n+1}^\TT)
			\end{bmatrix} \in\mathbb R^{q^2 \times n}
		\end{equation*}
		\vspace{-\baselineskip}\;%
		Compute weights $\gamma\gets\RRE(U)$\;
		Assemble low-rank factors of extrapolant:
		\begin{equation*}
			\widehat{Z} \gets Z_n,
			\quad
			\widehat{D} \gets \begin{bmatrix}
				\tau_1 D_1 \\
				& \tau_2 \tilde D_2 \\
				&& \ddots \\
				&&& \tau_n \tilde D_n
			\end{bmatrix}
			,\quad\text{where }
			\tau_i = \sum_{j=i}^n \gamma_j
		\end{equation*}
		\vspace{-\baselineskip}\;%
		\label{alg:lowrank_rre:extrapolant}
	}
\end{algorithm}

Algorithm~\ref{alg:lowrank_rre} summarizes the resulting procedure as an extension to Algorithm~\ref{alg:rre}.
Assuming the iterative process~\eqref{eqn:X=F(X) lr} reveals an incremental structure~\eqref{eqn:Xi_lr_decomp_inc},
Algorithm~\ref{alg:lowrank_rre} can be applied in any of the two \ac{RRE} modes shown in Algorithms~\ref{alg:noncycling_rre} and~\ref{alg:cycling_rre}.
The final ingredient is exploiting the incremental structure~\eqref{eqn:Xi_lr_decomp_inc}
to more efficiently compute the extrapolant~\eqref{eqn:matrre_extrapolant}
as shown in line~\ref{alg:lowrank_rre:extrapolant}.
Note that the first diagonal block of $\widehat{D}$
is given by~$D_1$,
the complete inner low-rank factor of~$X_1$.
Optionally,
one may remove the low-rank blocks~$V_i$ and~$\tilde D_i$ for which
$\tau_i := \sum_{j=i}^n \gamma_j$ is sufficiently close to zero:
\begin{equation}
	\abs{\tau_i}
	\leq \sqrt{\machineprecision} \max_{1\leq j\leq n} \abs{\tau_j}
	.
\end{equation}

\subsection{RRE for nonstationary fixed-point processes}\label{sct:flexible_rre}

Up to this point, we have considered \ac{RRE} in the context of \emph{stationary} fixed-point processes of the form~\eqref{eqn:xi_fp}.
By stationary, we mean that the fixed-point function $f$ stays the same for all iterations.
In this section, we consider \emph{nonstationary} fixed-point processes of the form~\eqref{eqn:xi_fp_iteration_dependent}.
To apply \ac{RRE} to nonstationary fixed-point processes, one could still use the weights obtained from the standard formulation~\eqref{eqn:rre_opt}.
However, as we will see later in this section, the strategy can lead to slow convergence.
Hence, we propose an alternative formulation and show its improved convergence speed.

%To demonstrate the difference between the standard formulation and the alternative formulation, we will consider a simple example of a vector-valued fixed-point process.
%Ultimately, we are interested in nonstationary fixed-point processes in the context of iterative solution of large-scale matrix equations.
%In this context, the iterates will be low-rank matrix factorizations instead of vectors.
%We will further discuss this setting in Section~\ref{sct:application}.
We now consider a simple example of vector-valued fixed-point process to demonstrate the difference between the standard formulation and the alternative formulation.
%Nonstationary fixed-point processes occur for instance when iteratively solving linear systems using different preconditioners at each step.
%The \ac{FGMRES} introduced in~\cite{saad1993flexible} is a generalization of \ac{GMRES} that is designed for this scenario.
As we have discussed above, an important difference between stationary and nonstationary fixed-point processes is that the two notions of the residual, equations~\eqref{eqn:residual:equation} and~\eqref{eqn:residual:iteration}, do no longer coincide.
Worse yet, the residual~\eqref{eqn:residual:iteration} of the iterative process becomes nonstationary as well.
That is,
\begin{equation}
	\reseq(x_i)
	:= b - Ax_i
	\neq x_{i+1} - x_i
	= M_i^{-1} (b - Ax_i)
	=: \resit_i(x_i)
	.
\end{equation}
Therefore, we must work with the residual of the underlying equation, \(\reseq\),
as shown in Algorithm~\ref{alg:nonstationary_rre} ($\RRE_{\mathscr F}$) as an extension of Algorithm~\ref{alg:rre}.

To illustrate the advantage of using \(\reseq\),
we present a numerical example involving the iterative solution of the linear system $Ax = b$ with the well-known \ac{SOR} scheme~\cite[Section~11.2.7]{golub2013matrix},
see formula~\eqref{eqn:xi+1=Axi+b} with
\begin{equation}
	% TODO: SOR me
	M = M_i = \frac{1}{\omega_i} D_A + L_A
	\qquad\text{and}\qquad
	N = N_i = \left( \frac{1}{\omega_i} - 1 \right) D_A - U_A,
	\label{eqn:fp_example}
\end{equation}
where $L_A$, $D_A$, and $U_A$ denote the strictly lower triangular, diagonal,
and strictly upper triangular parts of $A = L_A + D_A + U_A$, respectively,
and $0 < \omega_i < 2$.
For the underlying equation \(Ax=b\),
let $A \in \mathbb{R}^{20 \times 20}$ be a tri-diagonal matrix
with $1/10$ on the main diagonal and $1/20$ on the sub-diagonals,
and set $b = A\ones / \norm{A\ones}$.
We initialize the iterative scheme with $x_1 = 0$ and compare three runs:
\begin{enumerate}
	\item without extrapolation;\hfill($x_{i+1}=f(x_i)$)
	\item with cycling \ac{RRE} using the standard formulation~\eqref{eqn:rre_opt};\hfill($\RRE_\Delta$)
	\item\label{run:3} with cycling \ac{RRE} using formulation~\eqref{eqn:residual_extrapolant} but with~$\reseq(\cdot)$ in place of~$\resit(\cdot)$.\hfill($\RRE_{\mathscr F}$)
\end{enumerate}

\begin{figure}
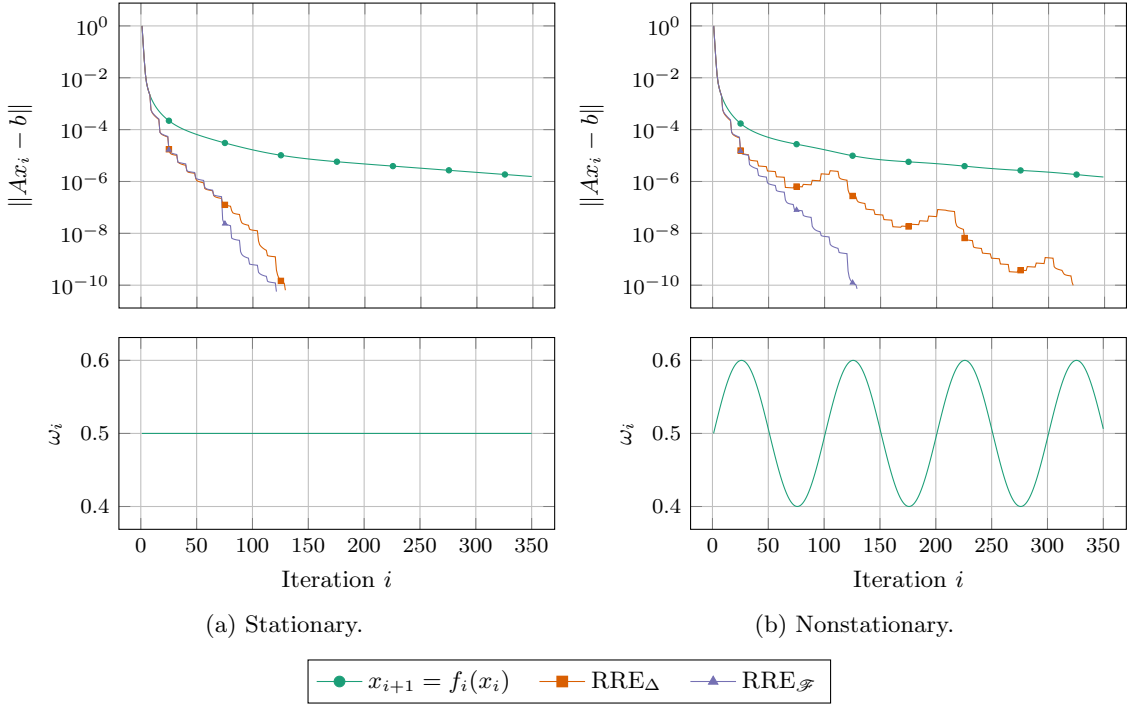

	\centering
	\begin{subfigure}[b]{0.5\textwidth}
		\simpleplot[
			legend style = {font = \small},
			legend columns = -1,
			legend entries = {
				$x_{i+1} = f_i(x_i)$,
				$\RRE_\Delta$,
				$\RRE_{\mathscr F}$,
			},
			legend to name = fig:example_stationary_fp:legend,
		]{simple.csv}
		\caption{Stationary.}%
        \label{fig:example_stationary_fp}
	\end{subfigure}%
	\begin{subfigure}[b]{0.5\textwidth}
		\centering
		\simpleplot{simple_nonstationary.csv}
		\caption{Nonstationary.}%
        \label{fig:example_nonstationary_fp}
	\end{subfigure}%
	\\
	\medskip
	\ref*{fig:example_stationary_fp:legend}% chktex 2
	\caption{Comparison of \ac{RRE} using the standard formulation and the proposed
      formulation on a stationary and nonstationary fixed-point process.}%
    \label{fig:example_stationary_vs_nonstationary_fp}
\end{figure}

The results for \ac{RRE} setting $n = 8$ are shown in Figure~\ref{fig:example_stationary_vs_nonstationary_fp} for
the stationary case ($\omega_i = 0.5$) and
the nonstationary case ($\omega_i = 0.5 + 1/10 \sin(0.02 \pi i)$).
In both scenarios,
using the \ac{RRE} that minimizes the residual of the underlying equation
accelerates convergence (fewer than 130 iterations). % 121 or 129 iterations
The standard formulation of \ac{RRE}, on the other hand,
encountered issues when the iterates are generated from an nonstationary fixed-point process.
Although the standard \ac{RRE} still converges in both cases and is even slightly ahead of our new formulation for the stationary case until iteration $i=72$,
it is significantly slower for the nonstationary process (329 iterations) than for the stationary process (129 iterations).

\begin{algorithm}
	\caption{%
		\ac{RRE} for non-stationary processes~\eqref{eqn:xi_fp_iteration_dependent}.
	}%
	\label{alg:nonstationary_rre}
	\KwIn{%
		Iterates $x_1, \ldots, x_n$.
	}
	\KwOut{%
		Extrapolant~$\widehat x$.
	}
	\medskip
	\Fn{$\RRE_{\mathscr F}(x_1, \ldots, x_n)$}{%
		Initialize matrix of residuals
		\begin{math}
			U \gets \begin{bmatrix}
				\mathscr F(x_1) & \dots & \mathscr F(x_n)
			\end{bmatrix}
			\in\mathbb R^{d\times n}
		\end{math}\;
		$\gamma\gets\RRE(U)$\;
		$\widehat x \gets \sum_{i=1}^n \gamma_i x_i$\;
	}
\end{algorithm}

Observe that $\RRE_{\mathscr F}$ does not require knowledge about iterate~$x_{n+1}$ to extrapolate from iterates~$x_1,\ldots,x_n$.
Therefore, Algorithms~\ref{alg:noncycling_rre} and~\ref{alg:cycling_rre} may be reformulated to compute an extrapolation with one less application of~$f$.
We leave this reformulation as an exercise for the reader.

\subsection{Extension to large-scale matrix equations}\label{sct:application}
The generalization of \ac{RRE} to low-rank matrices in Section~\ref{sct:matrix_rre} and to nonstationary fixed-point iterations in Section~\ref{sct:flexible_rre} may be used independently.
However, in the important case of iteratively solving large-scale matrix equations we shall require both extensions simultaneously.
Hence, in this section we will provide more details of the method in this usage scenario.

We focus on the following \ac{ARE} with solution $X \in \mathbb{R}^{d \times d}$:
\begin{equation}
	\mathscr{R}(X) := A^\TT X E + E^\TT X A + C^\TT C - E^\TT X B H^{-1} B^\TT X E = 0
	\in\mathbb{R}^{d\times d}
	,
	\label{eqn:riccati}
\end{equation}
where $B$ and $C^\TT$ are thin rectangular matrices and $H$ is symmetric positive definite.
The \ac{ARE}~\eqref{eqn:riccati} reduces to a linear Lyapunov equation whenever $B = 0$.
\acp{ARE}, and its adjoint variants which we do not consider here, are important nonlinear matrix equations that appear in many problems in control theory~\cite{Jungers2017historical}.
In such problems, typically one is only interested in the unique stabilizing solution $X$ such that the spectrum of $(A - B H^{-1} B^\TT X E, E)$ is in the open left half plane.
The stabilizing solution is guaranteed to exist if $(A, E, B)$ is stabilizable and $(A, E, C^\TT)$ is detectable~\cite{lancaster1995algebraic,BinIM12}.
\chdeleted{as cited by~\cite{benner2020numerical}.}

In the literature, there exist many iterative solvers that approximate the stabilizing solution of an \ac{ARE} with sparse coefficients.
A comprehensive overview and comparison of different solvers is given in~\cite{benner2020numerical}.
In that comparison study, one of the most performant solvers is the RADI method, first described in~\cite{benner2018radi}.
In this section,
we only give a very brief introduction to the RADI method since our focus is on how to accelerate RADI using the generalizations of \ac{RRE}
derived in the previous sections;
see~\cite{benner2018radi} for more details on the full version of RADI\@.

\begin{algorithm}[t]
	\caption{%
		RADI iteration map
		\cite[formula~(12)]{benner2018radi} % but modified; chktex 2
		\cite[\texttt{mess\_lrradi} simplified]{SaaKB-mmess-all-versions}. % but simplified; chktex 2
	}%
	\label{alg:radi}
	\KwIn{%
		Iterate $X_i = Z_i D_i Z_i^\TT$ given by
		$Z_i\in\mathbb{R}^{d\times z_i}$ and
		$D_i\in\mathbb{R}^{z_i\times z_i}$.
		\linebreak
		Residual $\mathscr{R}(X_i) = R_i T R_i^\TT$
		given by $R_i\in\mathbb{R}^{d\times r}$ and $T\in\mathbb{R}^{r\times r}$.
	}
	\KwOut{%
		Iterate $X_{i+1} = Z_{i+1} D_{i+1} Z_{i+1}^\TT$ and
		residual $\mathscr{R}(X_{i+1}) = R_{i+1} T R_{i+1}^\TT$.
	}
	\medskip
	\Fn{$F_i(Z_i, D_i, R_i, T)$}{%
		Select shift parameter $\sigma_i\in\RRneg$\;
		Compute outer low-rank factor of increment:
		\begin{equation*}
			V_{i+1} \gets
			\sqrt{-2\sigma_i} \cdot
			(A^\TT - E^\TT Z_i D_i Z_i^\TT B H^{-1} B^\TT + \sigma_i E^\TT)^{-1} R_i T
		\end{equation*}
		\vspace{-\baselineskip}\;%
		Compute inner low-rank factor of increment:
		\begin{equation*}
			\tilde Y_{i+1} \gets
			T - \frac{1}{2\sigma_i}
			\big( V_{i+1}^\TT B \big)
			H^{-1}
			\big( V_{i+1}^\TT B \big)^\TT
			\qquad\text{and}\qquad
			\tilde D_{i+1} \gets (\tilde Y_{i+1})^{-1}
		\end{equation*}
		\vspace{-\baselineskip}\;%
		Assemble low-rank factors of $X_{i+1}$:
		\begin{equation*}
			Z_{i+1} \gets \begin{bmatrix}
				Z_i & V_{i+1}
			\end{bmatrix}
			\qquad\text{and}\qquad
			D_{i+1} \gets \begin{bmatrix}
				D_i \\
				& \tilde D_{i+1}
			\end{bmatrix}
		\end{equation*}
		\vspace{-\baselineskip}\;%
		Update outer residual factor:
		\begin{equation*}
			R_{i+1} \gets R_i + \sqrt{-2\sigma_i}\cdot E^\TT V_{i+1}\tilde D_{i+1}
		\end{equation*}
		\vspace{-\baselineskip}\;%
	}
\end{algorithm}

The iterations of RADI can be described as follows.
For a given initialization $X_1 = Z_1 D_1 Z_1^\TT$,
write the initial residual as
\begin{math}
	\mathscr R(X_1) = R_1 T R_1^\TT,
\end{math}
where
\begin{equation}
\begin{alignedat}{2}
	R_1 &:= \begin{bmatrix}
		C^\TT & A^\TT Z_1 & E^\TT Z_1
	\end{bmatrix}
	&&\in\mathbb{R}^{d\times r}
	,
	\\
	T &:= \begin{bmatrix}
		I & \textbf{0} & \textbf{0} \\
		\textbf{0} & \textbf{0} & D_1 \\
		\textbf{0} & D_1 & - D_1Z_1^\TT B H^{-1} B^\TT Z_1 D_1
	\end{bmatrix}
	&&\in\mathbb{R}^{r\times r}
	,
\end{alignedat}
\label{eqn:riccati:residual}
\end{equation}
see, for example, \cite{BenLP08} for a comparison.
Subsequent iterates $X_2, X_3, \ldots$ (and \chadded{the corresponding} residuals) may then be computed using
the nonstationary RADI iteration map \chadded{shown in Algorithm}~\ref{alg:radi};
\chadded{cf.~\eqref{eqn:X=F(X) lr}, but with nonstationary~$F_i$.}
Note that complex shift parameters (with negative real part) can also be used in RADI, in which case complex arithmetic has to be used, e.g., in computing the low-rank factors~$V_{i+1}$ and~$\tilde Y_{i+1}$ of the increments.
\cite[Proposition~3]{benner2018radi} describes how to handle a pair of complex-conjugated shifts at once, % chktex 2
reducing the use of complex arithmetic,
which is implemented in~\cite[\texttt{mess\_lrradi}]{SaaKB-mmess-all-versions},
but, for simplicity, not reflected in Algorithm~\ref{alg:radi}.
However, Algorithm~\ref{alg:radi}
extends~\cite[formula~(12)]{benner2018radi} to nonzero initializations, $X_1\neq 0$.
Note that the residuals permit $\mathscr{R}(X_i) = R_i T R_i^\TT$,
where the inner factor~$T$ is fixed for all iterations $i\in\NN$.
%A quick check is that the algorithm reduces to \ac{ADI} if $T=I$ and to the algorithm for low-rank Lyapunov with nonzero initialization if $B=0$.

An attractive feature of RADI is that the rank of the residual is bounded from above by $r$.
The initialization $X_1$ fully determines the maximum rank $r$.
If the spectrum of $(A, E)$ is in the open left half plane, the initialization $X_1 = 0$ is typically used.
In this case,
the residual factors~\eqref{eqn:riccati:residual} collapse to $R_1=C^\TT$ and $T=I$ and,
hence, the rank of the residual is equal to $\rank(C^\TT C)=\rank(C)$.

Note that another useful property of RADI (as we mentioned in Section~\ref{sct:matrix_rre}) is that each iterate~$X_i$ can be written as a rank-$r$ update of~$X_{i-1}$.
In terms of formula~\eqref{eqn:Xi_lr_decomp_inc}, it is $v_{i+1} := r$ for all~$i\in\NN$.
Hence, the rank of the extrapolant~\eqref{eqn:matrre_extrapolant} is upper bounded by the rank of $X_n$.

The convergence speed of RADI can depend greatly on an appropriate selection of the shift parameters.
Many selection methods are proposed in literature, see, e.g.,~\cite{simoncini2016analysis,benner2018radi,benner2020numerical}.
However, regardless of the exact shift selection method that is used, the RADI iteration converges to the unique stabilizing positive semi-definite solution of the \ac{ARE}~\eqref{eqn:riccati}, e.g.,
if the spectrum of $(A, E)$ is in the open left half plane, $X_1 \succeq 0$, $\mathscr{R}(X_1) \succeq 0$ and the shift parameters $\sigma_i$ satisfy the non-Blaschke condition~\cite[Theorem~6.1]{massoudi2016analysis}:
\begin{equation}
	\sum_{i=1}^{\infty} \frac{ \sigma_i }{ 1 + \abs{\sigma_i}^2 } = -\infty.
	\label{eqn:nonblaschke_condition}
\end{equation}
Under this convergence guarantee, Algorithm~\ref{alg:radi}
may be viewed as a matrix version of the nonstationary fixed-point process~\eqref{eqn:xi_fp_iteration_dependent}:
\begin{equation}
	X_{i+1} = F_i(X_i)
	,
	\label{eqn:Xi_fp_nonstationary_matrix}
\end{equation}
with $F_i(X_i) = X_i + V_{i+1} \tilde{D}_{i+1} V_{i+1}^\TT$.
The nonstationarity in process~\eqref{eqn:Xi_fp_nonstationary_matrix} is the result of allowing different shift parameters~$\sigma_k$ in each iteration.
Hence, the results of Section~\ref{sct:flexible_rre} suggest that the residual of the \ac{ARE}~\eqref{eqn:riccati} should be used to find the \ac{RRE} coefficients.
That is, instead of solving the optimization problem~\eqref{eqn:matrre_opt},
the matrix version of \ac{RRE} formulated for stationary fixed-point iterations, we solve
\begin{equation}
	\gamma = \argminst{%
		g\in\mathbb R^n
	}{%
		\norm{ \sum_{i=1}^{n} g_i R_i T R_i^\TT }
	}{%
		\sum_{i=1}^{n} g_i = 1
		.
	}
	\label{eqn:matrre_opt_nonstationary}
\end{equation}
Recall that the Riccati operator~$\mathscr{R}(\cdot)$ is nonlinear,
such that this optimization problem only minimizes an approximation of the extrapolant's residual.

Optimization problem~\eqref{eqn:matrre_opt_nonstationary} involves the norm of a large but low-rank matrix,
similarly to~\eqref{eqn:matrre_opt_more_efficient:interim}.
Hence, analogous to the derivation of~\eqref{eqn:matrre_opt_more_efficient} via Lemma~\ref{lem:norm}, the same strategy applies here.
The resulting procedure is summarized in Algorithm~\ref{alg:nonstationary_lowrank_rre},
which combines Algorithms~\ref{alg:lowrank_rre} and~\ref{alg:nonstationary_rre}.
A visualization of the procedure is provided in Figure~\ref{fig:matrre_radi}.
Recall that, in order to generate an $n$-term extrapolant,
$\RRE_{\Delta}^{\LR}$ (Algorithm~\ref{alg:rre}) requires $n+1$~iterates while
$\RRE_{\mathscr F}^{\LR}$ (Algorithm~\ref{alg:nonstationary_rre}) requires~$n$.
Compare Figures~\ref{fig:lowrank_rre} and~\ref{fig:matrre_radi} for the case $n=2$,
which show that $\RRE_{\mathscr F}^{\LR}$ is able to use more information of the iterates than $\RRE_{\Delta}^{\LR}$ when given the same number of iterates.

Note that Algorithm~\ref{alg:nonstationary_lowrank_rre} is presented in the way to contrast Algorithm~\ref{alg:radi} and emphasize the application of \ac{RRE} therein, yet it could be easily extended to arbitrary functional $\mathscr F$,
namely, by dropping the assumption that each residual $\mathscr R(X_i)$ has the same rank.
Instead, all we need is to allow the residual factors $R_i$ and associated inner factor $T_i$
(which may no longer be the same for each iteration)
to have a different size and partition the QR-decomposition in Line~\ref{algline:nonstationary_lowrank_rre-QR} accordingly.
\chadded{%
Furthermore, Algorithm~\ref{alg:nonstationary_lowrank_rre} can easily be modified to use a different orthogonal-times-square factorization,
such as a pivoted QR decomposition.}

\begin{algorithm}[t]
	\caption{%
		\ac{RRE} for nonstationary low-rank matrix sequences~\eqref{eqn:matrre_opt_nonstationary}.
	}%
	\label{alg:nonstationary_lowrank_rre}
	\KwIn{%
		Iterates $X_1, \ldots, X_n$ given by
		$Z_1\in\mathbb{R}^{d\times z_1}$
		and
		$D_1\in\mathbb{R}^{z_1\times z_1}$
		as well as
		$V_i\in\mathbb{R}^{d\times v_i}$
		and
		$\tilde D_i\in\mathbb{R}^{v_i\times v_i}$
		fulfilling~\eqref{eqn:Xi_lr_decomp} and~\eqref{eqn:Xi_lr_decomp_inc}
		for $i=2,\ldots,n$.
		\linebreak
		Residuals $\mathscr{R}(X_i) = R_i T R_i^\TT$, where
		$R_i\in\mathbb{R}^{d\times r}$ and $T\in\mathbb{R}^{r\times r}$.
	}
	\KwOut{%
		Extrapolant $\widehat{X} = \widehat{Z} \widehat{D} \widehat{Z}^\TT$.
	}
	\medskip
	\Fn{$\RRE_{\mathscr R}^{\LR}(X_1, \ldots, X_n, R_1, \ldots, R_n, T)$}{
		Compute triangular factor of thin QR-decomposition:
		\begin{equation*}
			\begin{bmatrix}
				R_1 & \dots & R_n
			\end{bmatrix}
			= \tilde Q \begin{bmatrix}
				\tilde R_1 & \dots & \tilde R_n
			\end{bmatrix}
		\end{equation*}
		\vspace{-\baselineskip}\;\label{algline:nonstationary_lowrank_rre-QR}%
		Initialize matrix of residuals:
		\begin{equation*}
			U \gets\begin{bmatrix}
				\vec(\tilde R_1 T \tilde R_1^\TT) & \dots &
				\vec(\tilde R_n T \tilde R_n^\TT)
			\end{bmatrix} \in\mathbb R^{(nr)^2 \times n}
		\end{equation*}
		\vspace{-\baselineskip}\;%
		Compute weights $\gamma\gets\RRE(U)$\;
		Assemble low-rank factors of extrapolant:
		\begin{equation*}
			\widehat{Z} \gets Z_n,
			\quad
			\widehat{D} \gets \begin{bmatrix}
				\tau_1 D_1 \\
				& \tau_2 \tilde D_2 \\
				&& \ddots \\
				&&& \tau_n \tilde D_n
			\end{bmatrix}
			,\quad\text{where }
			\tau_i = \sum_{j=i}^n \gamma_j
		\end{equation*}
		\vspace{-\baselineskip}\;%
	}
\end{algorithm}

\begin{figure}[t]
	\centering
	\newcommand\ifresiduals\iftrue%
	\resizebox{0.9\linewidth}{!}{% chktex-file 1  % nodes without options are whitespace terminated.
                 % we could alternatively define \node as a safe command
% chktex-file 11 % foreach syntax uses the dots to describe sequences
% chktex-file 26 % singleton semicolon trigger this
% chktex-file 36 % tikz syntax needs brackets with non-whitespace char in front

% Whether to show residual blocks:
\providecommand\ifresiduals\iffalse

% Configure height, width, and separation between colorful blobs:
\setlength\myH{5em}%
\setlength\myW{2ex}%
\setlength\myS{1ex}%

% Low-rank factorization:
\newcommand\nodeVDV[3]{%
	\node (#1) {#2};
	\draw[fill, color=#3!50]
		(#1.center)
		++(-1.5\myW-\myS-0.5\myH,-0.5\myH) % center node in white space
		rectangle ++(\myW,\myH)
		++(\myS,-\myW)
		rectangle ++(\myW,\myW)
		++(\myS,-\myW)
		rectangle ++(\myH,\myW)
		;
	\coordinate (#1 north) at (current bounding box.north);
	\coordinate (#1 south) at (current bounding box.south);
}%

% Residual factor:
\newcommand\nodeR[3][1]{%
	\ifresiduals
	\node (#2) {#3};
	\draw[fill, gray!50]
		(#2.west) ++(-\myS,-0.5\myH)
		foreach \distance in {1,...,#1} {
			rectangle ++(-\myW,\myH)
			++(-\myS,-\myH)
		};
	\else
	\coordinate[yshift=.5ex] (#2);
	\fi
}%

\begin{tikzpicture}[
	thick,
	X/.style = {
		fill = gray!20,
		minimum size = \myH,
	},
	heading/.style = {
		above = \myS,
		inner sep = 0pt,
	},
	every node/.style = {
		anchor = base,
	},
	every outer matrix/.style = {
		inner sep = 0pt,
	},
]
	\matrix (foo) [
		row sep = 2\myS,
		column sep = \myS,
	] {
		% Row 1:
		\node[X] (X1) {$X_1$}; &
		\node {=}; &
		\nodeVDV{V11}{$\hphantom{\tau_1} Z_1 D_1 Z_1^\TT$}{mycolor1} &&&&& [2\myS]
		\nodeR{R1}{$R_1$} & [4\myS]
		\coordinate (R1+); \\

		% Row 2:
		\node[X] (X2) {$X_2$}; &
		\node {=}; &
		\nodeVDV{V21}{$\hphantom{\tau_1} Z_1 D_1 Z_1^\TT$}{mycolor1} &
		\node {+}; &
		\nodeVDV{V22}{$\hphantom{\tau_2} V_2 \tilde D_2 V_2^\TT$}{mycolor2} &&&
		\nodeR{R2}{$R_2$} &
		\coordinate (R2+); & [4\myS]
		\node[draw, fill=gray!20] (RRE) {RRE}; \\

		% Row 3:
		\node[X] (X3) {$X_3$}; &
		\node {=}; &
		\nodeVDV{V31}{$\hphantom{\tau_1} Z_1 D_1 Z_1^\TT$}{mycolor1} &
		\node {+}; &
		\nodeVDV{V32}{$\hphantom{\tau_2} V_2 \tilde D_2 V_2^\TT$}{mycolor2} &
		\node {+}; &
		\nodeVDV{V33}{$\hphantom{\tau_3} V_3 \tilde D_3 V_3^\TT$}{mycolor3} &
		\nodeR{R3}{$R_3$} &
		\coordinate (R3+); \\

		% Row 4:
		\node[X, label={[heading] Extrapolant}] {$\widehat X_{\ifresiduals 3\else 2\fi}$}; &
		\node{=}; &
		\nodeVDV{V1}{$\tau_1 Z_1 D_1 Z_1^\TT$}{mycolor1} &
		\node {+}; &
		\nodeVDV{V2}{$\tau_2 V_2 \tilde D_2 V_2^\TT$}{mycolor2} &
		\ifresiduals
		\node {+}; &
		\nodeVDV{V3}{$\tau_3 V_3 \tilde D_3 V_3^\TT$}{mycolor3} &
		\else
		&&
		\fi
		\coordinate[yshift=.5ex] (RE); \\
	};

	% Draw column headers:
	\node[heading] at (V11 north -| X1) {Iterate};
	\node[heading] at (V11 north -| V2) {Low-rank factors};
	\ifresiduals
	\node[heading] at (V11 north -| R1) {Residual factors};
	\fi

	% Draw dashed lines in between rows:
	\foreach \row in {1, 2, 3} {
		\draw[
			dashed,
			transform canvas = { yshift = -\myS },
			shorten > = -\myS,
		] (X\row.west |- V\row1 south) ++(-\myS,0) -- (V\row1 south -| R\row.east);
	}

	% Draw RRE apparatus:
	\coordinate (junction) at (R1 -| R1+); % baseline adjustment
	\draw[->] (R1) -| (RRE);
	\draw (R2) -- (R2 -| junction);
	\draw (R3) -| (junction);
	\draw[->] (RRE) |- (RE);
	\node[anchor=south east] at (RE -| RRE) {$\tau_1, \tau_2\ifresiduals, \tau_3\fi$};
\end{tikzpicture}%
}
	\caption{%
		Schema illustrating the application of \ac{RRE} to the first 3 low-rank
		iterates $X_1, X_2, X_3$ and corresponding residual factors $R_1, R_2, R_3$ (with $T = I$),
		where the weights $\tau_1, \tau_2, \tau_3$ are as in~\eqref{eqn:matrre_tau}.
	}%
	\label{fig:matrre_radi}
\end{figure}
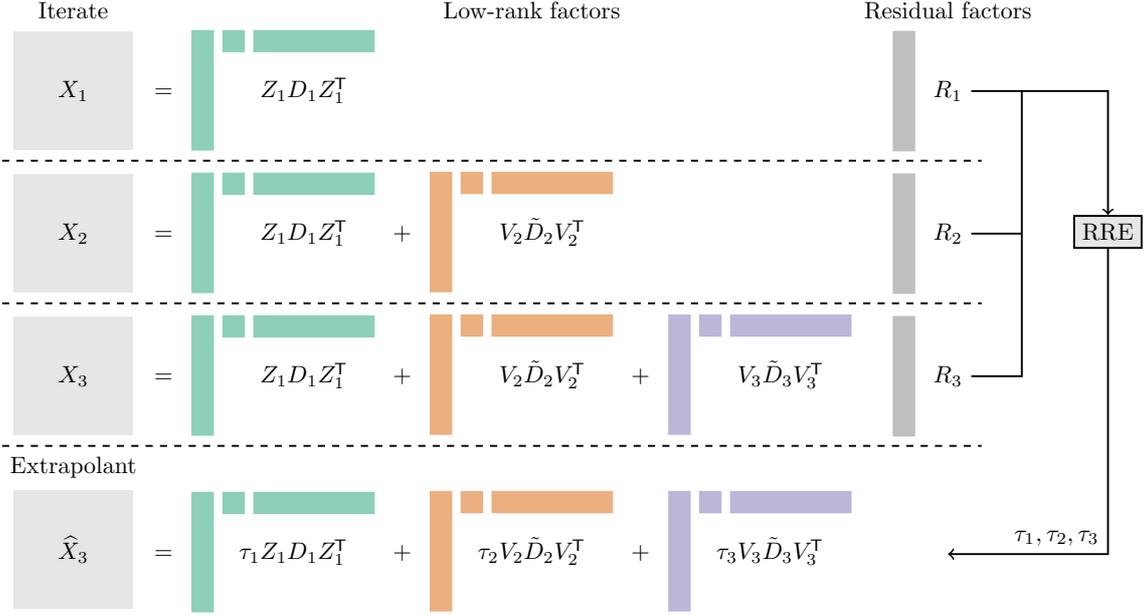

\newcommand\Landau{\mathcal O}
\begin{remark}
\chadded{%
The residual factors $R_1,\ldots,R_n$ and the inner factor $T$ required to formulate the optimization problem~\eqref{eqn:matrre_opt_nonstationary} are computed at each RADI step (Algorithm~\ref{alg:radi}) and are therefore available at no extra cost.
The additional cost of one \ac{RRE} extrapolation on top of the RADI iteration is dominated by the thin QR factorization of the concatenated residual factors,
$[R_1 \ \cdots \ R_n] \in \mathbb{R}^{d \times nr}$,
while the orthogonal factor need not be formed explicitly; cf. Algorithm~\ref{alg:nonstationary_lowrank_rre}.
This QR step costs $\Landau(d (nr)^2)$~\acp{FLOP},
where $\Landau(\cdot)$ denotes Landau notation.
Forming the matrix~$U\in\mathbb R^{(nr)^2 \times n}$ requires $\Landau(n^3r^3)$~\acp{FLOP}.
The subsequent QR-based computation of the \ac{RRE} coefficients using Algorithm~\ref{alg:rre} contributes $\Landau((nr)^2 n^2)$~\acp{FLOP},
where the subsequent triangular solves and normalization only add lower-order $\Landau(n^2)$ work.
Hence, the per-extrapolation overhead is $\Landau(d n^2 r^2 + n^4 r^2)$~\acp{FLOP}.
In terms of storage, the extra memory beyond the baseline RADI iteration is the residual window itself, requiring $\Landau(dnr)$ floating-point numbers,
together with $n^3 r^2$ elements for storing the matrix~$U$.}
\end{remark}

To monitor progress and provide a stopping criterion, the residual of the
extrapolant, \(\mathscr{R}(\widehat{X})\), is needed.
Since the extrapolant has a low-rank structure,
$\widehat{X} = \widehat{Z} \widehat{D} \widehat{Z}^\TT$,
the corresponding residual~$\R(\widehat{X})$ directly reveals a similar low-rank structure;
see formula~\eqref{eqn:riccati:residual}.
Lemma~\ref{lem:norm} may again be used to efficiently evaluate the Frobenius or spectral norm.
Interestingly, as the following theorem will show,
$\range(\R(\widehat{X})) \subseteq \range(R_1) \cup\dots\cup \range(R_n)$,
for which the proof naturally provides an efficient approach to recursively obtain a factorization of $\mathscr{R}(\widehat{X})$. The proof is induction-based and provided in Appendix~\ref{append:proof-residual-expression}.

\begin{thm}\label{thm:residual-expression}
	Under the consistent notations in Algorithm~\ref{alg:radi} and Algorithm~\ref{alg:nonstationary_lowrank_rre}, let $\{X_i\}_{i=1}^{n}$ denote the low-rank matrix sequence generated for solving the ARE~\eqref{eqn:riccati}. Then the residual of the $n$-term \ac{RRE} extrapolant $\widehat{X}$, produced by Algorithm~\ref{alg:nonstationary_lowrank_rre}, remains in the range of $[R_1\ R_2\ \dots\ R_n]$, and there exists the factorization
	\[
	\mathscr{R}(\widehat{X}) = [R_1\ R_2\ \dots\ R_n]\mathscr{H}_n[R_1\ R_2\ \dots\ R_n]^\TT,
	\]
	of this residual, where $R_n$ is the residual factor of $\mathscr{R}(X_n)=R_n T R_n^\TT$
	and $\mathscr{H}_n$ is a symmetric matrix of size $np\times np$ that satisfies
	\begin{equation*}
		\mathscr{H}_n\equiv  \mathscr{H}_n(\rr_1,\rr_2,\dots,\rr_n) =
		\begin{cases}
			\begin{bmatrix}
				\rr_1^2T+(\rr_2-\rr_2^2)\Y_2 & (\rr_2-\rr_2^2)(T-\Y_2) \\
				(\rr_2-\rr_2^2)(T-\Y_2)    &  \rr_2^2 T +(\rr_2-\rr_2^2)\Y_2
			\end{bmatrix}, & n=2, \vspace{5pt} \\
			\mathscr{A}_n + \mathscr{B}_n + \mathscr{C}_n,\quad
			& n\ge 3,
		\end{cases}
	\end{equation*}
	where
	\begin{equation*}
		\mathscr{A}_n = \blkdiag
		\left(\mathscr{H}_{n-1}(\rr_1,\rr_2,\dots,\rr_{n-2},\rr_{n-1}+\rr_{n}), \textbf{0}_{p\times p}\right),
	\end{equation*}
	\begin{equation*}
		\mathscr{B}_n = \blkdiag\left(\textbf{0}_{(n-2)p\times (n-2)p},
		\begin{bmatrix}
			(\rr_n^2-2\rr_n)T+(\rr_n-\rr_n^2)\Y_n    &   (\rr_n-\rr_n^2)(T-\Y_n) \\
			(\rr_n-\rr_n^2)(T-\Y_n) & \rr_n^2 T +(\rr_n-\rr_n^2)\Y_n
		\end{bmatrix}\right),
	\end{equation*}
	and
	\begin{equation*}
		\mathscr{C}_n = \widetilde{\mathscr{C}}_n + \widetilde{\mathscr{C}}_n^\TT,\quad
		\widetilde{\mathscr{C}}_n =
		\begin{bNiceArray}[
			first-row,code-for-first-row=\scriptstyle,
			last-col,code-for-last-col=\scriptstyle,
			]{ccc}
			(n-2)p       &  p           &  p            &  \\
			\textbf{0}   & M_n^d - M_n^t    & M_n^t-M_n^d         & (n-1)p \\
			\textbf{0}   & \textbf{0}   & \textbf{0}    & p
		\end{bNiceArray}\ ,
	\end{equation*}
	where
	\begin{align*}
		M_n &= [
			\beta_{n2}\alpha_{2n}C_{2n}^\TT \;\;
			\beta_{n3}\alpha_{3n}C_{3n}^\TT \;\;
			\dots \;\;
			\beta_{n,n-1}\alpha_{n-1,n}C_{n-1,n}^\TT
		]^\TT,
		\\
		M_n^d &= \begin{bmatrix}
			M_n^\TT &
			\textbf{0}
		\end{bmatrix}^\TT,
		\\
		M_n^t &= \begin{bmatrix}
		\textbf{0} &
		M_n^\TT
		\end{bmatrix}^\TT,
	\end{align*}
	and
	\begin{equation*}
		\alpha_{in}=
		\frac{1}{2\sqrt{\sigma_{i-1}\sigma_{n-1}}}\in\mathbb{R},\quad
		\beta_{ni}=\rr_n\sum_{j=1}^{i-1}\rr_j,\quad
		C_{in} = V_i^\TT BH^{-1}B^\TT V_n \in\mathbb{R}^{p\times p}.
	\end{equation*}
\end{thm}

\begin{remark}
  The result of Theorem~\ref{thm:residual-expression} provides an expression for the residual of the extrapolant in terms of the \ac{RRE} weights $\gamma_i$, $i = 1, \ldots, n$.
Thus, one could in principle directly minimize the Frobenius or spectral norm of this residual (applying Lemma~\ref{lem:norm} to ensure computational feasibility).
However, we note that Theorem~\ref{thm:residual-expression} treats the specific case of the \ac{ARE}, while the \ac{RRE} extrapolation technique we propose in this work can be applied to \emph{any} iterative process for which the residuals are available.
Furthermore, finding the \ac{RRE} weights~\eqref{eqn:matrre_opt_more_efficient} can be done even if the coefficients of the \ac{ARE} are not directly available (making this method less intrusive).
\end{remark}

We now consider the application of the \ac{RRE} extrapolant based on~\eqref{eqn:matrre_opt_efficient} in a cycling scheme (Algorithm~\ref{alg:cycling_rre}).
RADI is guaranteed to converge to the unique stabilizing solution of the \ac{ARE} if
(i) the shift parameters satisfy the non-Blaschke condition~\eqref{eqn:nonblaschke_condition}; % chktex 10
(ii)\label{item:psd_residual} it is initialized using a stabilizing initial guess $Y$ with $\mathscr{R}(Y) \succeq 0$ (in case the spectrum of $(A, E)$ is in the open left half plane, $Y = 0$ may be used); % chktex 10
and, in case the spectrum of $(A, E)$ is not in the open left half plane,
(iii) RADI is applied to the closed-loop \ac{ARE} associated with $Y$~\cite{massoudi2016analysis}. % chktex 10
Property (ii) is mentioned in~\cite[p.~308]{benner2018radi} but without an explanation.
We will investigate the conditions on the \ac{RRE} coefficients $\gamma_1, \ldots, \gamma_n$, for which these constraints on $Y$ and $\mathscr{R}(Y)$ are satisfied at the beginning of each cycle (i.e., after each extrapolation step).
For notational compactness,
we derive these conditions for $i=1,\dots,n$ as in Section~\ref{sct:preliminary},
and define $V_1 := Z_1$ as well as $\tilde D_1 := D_1$.

To investigate when $\widehat{X} \succeq 0$, we note that each iterate $X_i$ can be written as a sum of positive semi-definite products $V_j D_j V_j^\TT$:
\begin{equation}
	X_i = \sum_{j=1}^{i} V_j D_j V_j^\TT.
\end{equation}
Hence, a sufficient condition that $\widehat{X} = \sum_{i=1}^{n} \gamma_i X_i = \sum_{i=1}^{n} \tau_i V_i \tilde D_i V_i^\TT \succeq 0$, is $\tau_i \geq 0$, $i = 1, \ldots, n$.
This gives the following set of linear inequality constraints on $\gamma_i$, $i = 1, \ldots, n$, to ensure $\widehat{X} \succeq 0$:
\begin{equation}
	\sum_{i=j}^{n} \gamma_j \geq 0, \quad j = 1, \ldots, n.
	\label{eqn:matrre_gamma_bounds_X1}
\end{equation}

For the condition $\mathscr{R}(\widehat{X}) \succeq 0$, the next theorem gives a sufficient and necessary condition for the residual of the two-term extrapolant ($n = 2$) to be positive semi-definite.
Its proof is given in Appendix~\ref{append:proof-residual-psd-cond}.
\begin{thm}\label{thm:residual-psd-cond}
Under the same notations as in Theorem~\ref{thm:residual-expression}, if the inner residual factor $T$ is positive semi-definite, then the residual of the two-term extrapolant
\begin{equation*}
	\R(\rr_1X_1 + \rr_2X_2) =
	\begin{bmatrix}
		R_1 \ R_2
	\end{bmatrix}\mathscr{H}_2\begin{bmatrix}
		R_1^\TT \\ R_2^\TT
	\end{bmatrix}
\end{equation*}
is positive semi-definite if and only if $0\le \rr_1,\rr_2\le 1$.
\end{thm}

Although Theorem~\ref{thm:residual-psd-cond} only treats the case $n = 2$, we numerically observed (for an instance of the \ac{ARE} using randomly generated coefficients) that a similar condition holds for $n = 3$.
\chadded{More precisely,} when $T\succeq 0$, \chadded{we have} $\mathscr{R}(\widehat{X}) \succeq 0$ if and only if $0 \leq \rr_i \leq 1$, $i = 1, 2, 3$.
Based on these results, we formulate the following conjecture.

\begin{conjecture}\label{conj:are_rre_residual}
Under the same notations as in Theorem~\ref{thm:residual-expression}, if the inner residual factor $T$ is positive semi-definite, then the residual of the extrapolant
\begin{equation*}
	\R(\rr_1X_1 + \cdots + \rr_n X_n) =
	\begin{bmatrix}
		R_1 \ \dots \ R_n
	\end{bmatrix}\mathscr{H}_n\begin{bmatrix}
		R_1^\TT \\ \vdots \\ R_n^\TT
	\end{bmatrix}
\end{equation*}
is positive semi-definite if and only if
\begin{equation}
	0\le \rr_i \le 1, \quad i = 1, \ldots, n.
	\label{eqn:matrre_gamma_bounds_res}
\end{equation}
\end{conjecture}

The bounds~\eqref{eqn:matrre_gamma_bounds_res} to ensure positive semi-definiteness of the residual are stronger than the bounds~\eqref{eqn:matrre_gamma_bounds_X1} to ensure positive semi-definiteness of the extrapolant.
As a result, applying \ac{RRE} to \ac{ARE} in a restarting scheme leads to a search space for the coefficients that is more limited than for a non-restarting scheme (in which case~\eqref{eqn:matrre_gamma_bounds_res} is not required).
Hence, the residual of the extrapolant may also be larger than without the limitation~\eqref{eqn:matrre_gamma_bounds_res}, leading to an increased number of iterations and computational time required until convergence.
This matter will be numerically investigated in Section~\ref{sct:numerical_examples}.

The bounds on the \ac{RRE} coefficients~\eqref{eqn:matrre_gamma_bounds_X1} and~\eqref{eqn:matrre_gamma_bounds_res} are required because the solution of an \ac{ARE} is in general non-unique and one is typically only interested in the maximum (in the Loewner ordering) \ac{PSD} solution.
These bounds prevent the extrapolated sequence from converging to a non-\ac{PSD} solution.
In the special case of Lyapunov equations (when $B = 0$), the solution is unique and, hence, these bounds are not required.
Furthermore, the bounds~\eqref{eqn:matrre_gamma_bounds_res} ensuring the residual remains \ac{PSD} are not required when using a non-restarting scheme.

\begin{remark}\label{remark:convergence_psd_solution}
  Although the proposed bounds prevent the extrapolated sequence from converging to a non-\ac{PSD} solution, we cannot guarantee convergence to the desired \ac{PSD} solution.
  Finding convergence guarantees for \ac{RRE} applied on nonlinear systems is difficult and usually requires the introduction of heuristic assumptions, see~\cite{sidi2020convergence} for more details and some results.
\end{remark}

\section{Numerical examples}\label{sct:numerical_examples}
In this section, we evaluate the proposed method on several numerical examples of large-scale matrix equations with sparse coefficients.
More specifically, each example is associated with a first-order state-space model of the form
\begin{equation}
\begin{aligned}
	E \dot{z}(t) & = A z(t) + B u(t), \\
		    y(t) & = C z(t),
\end{aligned}
\label{eqn:ss}
\end{equation}
where $z(t) \in \mathbb{R}^{d}$ denotes the state, $u(t) \in \mathbb{R}^{p}$ is the input and $y(t) \in \mathbb{R}^{q}$ is the output
(in Section~\ref{sct:triplechain}, we consider a second-order system which we represent in the above first-order form).
We focus on examples in which the state dimension is large compared to the dimension of the input and output.
The operators $E, A, B$ and $C$ are matrices, where $E, A \in \mathbb{R}^{d \times d}$ are sparse, $E$ is regular and the spectrum of $(A, E)$ is in the open left half plane.

Associated with each model~\eqref{eqn:ss} is the \ac{ARE}~\eqref{eqn:riccati} using the coefficients of the state-space model~\eqref{eqn:ss} as the coefficients of the equation and $H = I_p$ unless otherwise specified.

We use the RADI method to produce iterates $X_1, X_2, \ldots$ for the \ac{ARE}.
We stopped the iterative process once the relative 2-norm of the residual reached the threshold $10^{-10}$.
Unless otherwise indicated, the extrapolation is performed with $n = 3$, where $n$ denotes the number of iterates used in the extrapolation.

All numerical experiments were performed on a compute server equipped with an AMD EPYC 7763 64-core processor and with 512~GiB of Random Access Memory (RAM). The numerical examples were run using 8~cores and required at most approximately 70~GiB of RAM. %chktex 8
We used MATLAB R2020b on Linux and RADI based on version 3.1 of the MATLAB version of the M.E.S.S. toolbox~\cite{SaaKB-mmess-all-versions}.
We also use RADI when iteratively solving the Lyapunov equations in Sections~\ref{sct:toeplitz} and~\ref{sct:rail} (in which case the RADI method coincides with the ADI method).
To generate the shift parameters, projection shifts~\cite{benner2014self} are used.
Note that the concept and usage of projection-based shifts generalize directly from linear matrix equations to \ac{ARE}s, see, e.g., \cite{benner2018radi,benner2020numerical}.
We used the function \texttt{mess\_res2\_norms} to \chadded{estimate the residual norm $\lVert\R(\widehat{X})\rVert_2$ by means of a Lanzcos process.
The scaled residual norms are plotted} associated with each of the following modes:
\begin{enumerate}
	\item RADI\@: The RADI iterates. Indicated in plots by a solid green disk if initialized from zero and by a green circle if initialized from a cycle restart.
	\item RADI+RRE (non-cycling): The RADI iterates accelerated by non-cycling \ac{RRE} (Algorithm~\ref{alg:noncycling_rre}).
	\item RADI+RRE (cycling): The RADI iterates accelerated by cycling \ac{RRE} (Algorithm~\ref{alg:cycling_rre}).
		After a cycle restart, the residuals are re-assembled akin to formula~\eqref{eqn:riccati:residual}.
\end{enumerate}
\chadded{In both RRE modes we enforce bound~\eqref{eqn:matrre_gamma_bounds_X1} on the RRE coefficients to ensure a positive semi-definite solution.}

In our implementation, the main computational cost of \ac{RRE} is the QR-decomposition of the residual factors (Line~\ref{algline:nonstationary_lowrank_rre-QR} in Algorithm~\ref{alg:nonstationary_lowrank_rre}).
We use the MATLAB function \texttt{qr} to accomplish this.
However, the Q-term of the QR-decomposition is not needed, which may be exploited to reduce the computational cost~\cite{BenLP08}.
Furthermore, we compute the \ac{RRE} extrapolant in each iteration (provided at least $n$ iterates are available).
Computing the extrapolant only when significant speed-up is expected could substantially lower the computational overhead.
Instead of  investigating such optimizations of the implementation, we focus here on the acceleration in terms of the number of iterations; see Appendix~\ref{append:runtime} for a brief discussion on runtimes.

\subsection{Example 1: Toeplitz system matrix}\label{sct:toeplitz}
\subsubsection{Nonlinear}
This is a synthetic example based on the modifications of~\cite[Example~1]{lin2015new} introduced in~\cite[Example~4.4]{benner2020numerical}.
It is given by the following matrices:
\begin{equation}
A=\left[\begin{array}{ccccccc}
2.8 & 1 & 1 & 1 & 0 & \dots & \\
-1 & 2.8 & 1 & 1 & 1 & 0 & \ddots \\
0 & -1 & 2.8 & 1 & 1 & 1 & \ddots \\
\vdots & \ddots & \ddots & \ddots & \ddots & \ddots & \ddots \\
& & & & 0 & -1 & 2.8
\end{array}\right], \quad E=I,
\end{equation}
with $d = \num{100,000}$.
The entries of the matrices $B$ and $C$ are drawn independently from a normal distribution
(together with the MATLAB command \lstinline{ rng(1) } to ensure reproducibility of the results).  % chktex 36
$B$ is scaled such that $\norm{B} = 1$.
We take $p = 5$ and try several values of $q$ (1, 20 and 40) to investigate the influence of this parameter on the results.
(Recall that $p$ and $q$ are the dimensions of the input and output.)
We set $H = 10^{-4} \cdot I_p$.

The results are shown in Figure~\ref{fig:results_toeplitz}.
Note first the staircase pattern present in all plots, in which a decrease in residual is sometimes followed by several iterations in which the residual stays almost constant.
For $q = 1$, i.e., the $C$ matrix has a single row, the improvement in residual from the extrapolation is modest for the first 20 iterations after which it is significantly improved.
As a result, the convergence criterion (the horizontal line in the figures) is reached already after 30 iterations instead of 44 iterations for the non-accelerated sequence.
In cycling mode, the 2 restarts in iterations 20 and 28 further reduce the residual, although the convergence criterion is still only reached after 30 iterations (the same as for the non-cycling mode).

When we increase $q$, it can be observed from Figures~\ref{fig:results_toeplitz_q20} and~\ref{fig:results_toeplitz_q40} that the number of iterations required by RADI increases to more than 60.
In non-cycling mode, for $q = 20$ \ac{RRE} reduces the number of iterations from 64 to just 63,
and for $q = 40$ it decreases the number of iterations from 76 to 64.

\begin{remark}~\label{remark:cycling_mode_stagnation}
In cycling mode, we observe that the iterates generated by RADI initialized from the new cycle have a non-PSD residual.
We suspect that this causes the stagnation of both the RADI iterates and the \ac{RRE} extrapolates from iteration 55.

Furthermore, as we did not implement condition~\eqref{eqn:matrre_gamma_bounds_res},
we observed the residual to be indefinite on every cycle restart in all of Section~\ref{sct:numerical_examples}.
It remains inconclusive whether this violation of property~(ii),
see page~\pageref{item:psd_residual},
has an impact on the convergence of RADI\@.
\end{remark}

During the first iterations, when the relative residual is close to 1, the residual after \ac{RRE} is for many iterations larger than the residual of RADI\@.
This is because the \ac{ARE} is a nonlinear equation and, hence, optimization~\eqref{eqn:matrre_opt} only approximately minimizes the residual of the extrapolant; see Section~\ref{sct:preliminary} but with $\reseq(\cdot)\neq\resit(\cdot)$.
The approximation becomes more accurate for smaller residuals, which is reflected by the fact that the \ac{RRE} residual is much smaller than the RADI residual for those iterations.

\begin{figure}[t]
	\centering
	\newcommand\common{
		xmax = 80,
		ymax = 10,
		ymin = 1e-14,
		ytickten = {-15,-12,...,0}, % must contain 0; chktex 11
	}
	\begin{subfigure}{0.5\textwidth}
		\centering
		\rreplot{benner2020_q1_nonlinear}{%
			legend entries={RADI, RADI+RRE (non-cycling), RADI+RRE (cycling)},
			legend to name=fig:results_toeplitz:legend,
			\common,
		}
		\caption{$q = 1$.}%
        \label{fig:results_toeplitz_q1}
	\end{subfigure}% important
	\hfill
	\begin{subfigure}{0.5\textwidth}
		\centering
		\rreplot{benner2020_q20_nonlinear}{\common}
		\caption{$q = 20$.}%
        \label{fig:results_toeplitz_q20}
	\end{subfigure}% important
	\\
	\bigskip
	\begin{subfigure}[c]{0.5\textwidth}
		\centering
		\rreplot{benner2020_q40_nonlinear}{\common}
		\caption{$q = 40$.}%
        \label{fig:results_toeplitz_q40}
	\end{subfigure}% important
	\hfill
	\ref*{fig:results_toeplitz:legend}% chktex 2
	\hspace*{\fill}
	\caption{%
		Results of the nonlinear Toeplitz example with varying numbers of outputs \(q\).
		\rreplotcommoncaption}%
    \label{fig:results_toeplitz}
\end{figure}

\subsubsection{Linear}
The previous results illustrate the ineffectiveness of the cycling strategy for the nonlinear \ac{ARE}
(see Remark~\ref{remark:cycling_mode_stagnation}).
In this section, therefore, we study a linear equation, for which RADI coincides with the ADI method.
Namely, the Lyapunov equation obtained by setting $B = 0$.
Because the spectrum of $(A, E)$ is in the open left half plane,
the solution is unique (in contrast to the nonlinear \ac{ARE}).
This property may result in improved effectiveness of the cycling strategy.

The results are shown in Figure~\ref{fig:results_toeplitz_linear}.
For $q = 1$, the effect of \ac{RRE} is negligible, indicating optimality of the ADI iterates.
For $q = 20$ and $q = 40$, on the other hand, both non-cycling and cycling \ac{RRE} reduce the number of iterations required to reach the convergence threshold.
In this linear setting, no stagnation of the cycling strategy is observed.

\begin{figure}[t]
	\centering
	\newcommand\common{
		xmax = 30,
		ymax = 10,
		ymin = 1e-14,
		ytickten = {-15,-12,...,0}, % must contain 0; chktex 11
	}
	\begin{subfigure}{0.5\textwidth}
		\centering
		\rreplot{benner2020_q1_linear}{%
			legend entries={RADI, RADI+RRE (non-cycling), RADI+RRE (cycling)},
			legend to name=fig:results_toeplitz_linear:legend,
			\common,
		}
		\caption{$q = 1$.}%
		\label{fig:results_toeplitz_q1_linear}
	\end{subfigure}% important
	\hfill
	\begin{subfigure}{0.5\textwidth}
		\centering
		\rreplot{benner2020_q20_linear}{\common}
		\caption{$q = 20$.}%
		\label{fig:results_toeplitz_q20_linear}
	\end{subfigure}% important
	\\
	\bigskip
	\begin{subfigure}[c]{0.5\textwidth}
		\centering
		\rreplot{benner2020_q40_linear}{\common}
		\caption{$q = 40$.}%
		\label{fig:results_toeplitz_q40_linear}
	\end{subfigure}% important
	\hfill
	\ref*{fig:results_toeplitz:legend}% chktex 2
	\hspace*{\fill}
	\caption{%
		Results of the linear Toeplitz example with varying numbers of outputs \(q\).
		\rreplotcommoncaption}%
	\label{fig:results_toeplitz_linear}
\end{figure}

\subsection{Example 2: chip}\label{sct:chip}
This example concerns a model of the convective flow in a microelectronic chip structure~\cite{morMooRGetal04}.
The 3D structure is semi-discretized using \ac{FEM} with $d = \num{20,082}$ states representing temperature.
The $q = 5$ outputs of the model correspond to temperatures at 5 points in the geometry.
The single input ($p = 1$) represents the power dissipated by a heat element in the structure.
The matrices have been downloaded from the MOR Wiki website~\cite{morwiki_convection}.
Due to the convection present in the problem, $A \neq A^\TT$, and the $E$ matrix is symmetric and has full rank.
In the experiments, we test several variations of the \ac{ARE}~\eqref{eqn:riccati} where we scale the quadratic term therein, i.e.,
\begin{equation}
	H = \lambda I_p.
\end{equation}
Increasing $\lambda$ makes the corresponding \ac{ARE} closer to a Lyapunov equation, while decreasing $\lambda$ enhances the nonlinearity of the equation.
The results for several values of $\lambda$ are displayed in Figure~\ref{fig:results_chip}.

We can observe that the benefit of \ac{RRE} is greater for smaller $\lambda$.
Thus, as the quadratic term becomes less significant, the iterates generated by RADI become closer to optimal (in the sense of optimization problem~\eqref{eqn:matrre_opt}).
For $\lambda = 10^{-4}$ (Figure~\ref{fig:results_chip_lambda0.0001}), the acceleration is greatest.
In iteration 41, the \ac{RRE} residual is more than 3 orders of magnitude better than the RADI residual.
Hence, we also start a new cycle at that iteration.
From that iteration onwards, the cycling scheme exhibits a lower residual than the non-cycling scheme, illustrating the potential benefit of the former.
However, the convergence threshold of $10^{-10}$ is not reached; similar to the Toeplitz system matrix example in Figure~\ref{fig:results_toeplitz_q20} (Remark~\ref{remark:cycling_mode_stagnation}).

\begin{figure}[t]
	\centering
	\newcommand\common{
		xmax = 70,
		ymin = 5e-12,
		ytickten = {-12,-9,...,0}, % must contain 0; chktex 11
	}
	\begin{subfigure}{0.5\textwidth}
		\centering
		\rreplot{chip_lambda1}{\common}
		\caption{$\lambda = 10^{0}$.}%
		\label{fig:results_chip_lambda1}
	\end{subfigure}% important
	\hfill
	\begin{subfigure}{0.5\textwidth}
		\centering
		\rreplot{chip_lambda0.01}{\common}
		\caption{$\lambda = 10^{-2}$.}%
		\label{fig:results_chip_lambda0.01}
	\end{subfigure}% important
	\\
	\bigskip
	\begin{subfigure}[c]{0.5\textwidth}
		\centering
		\rreplot{chip_lambda0.0001}{
			legend entries={RADI, RADI+RRE (non-cycling), RADI+RRE (cycling)},
			legend to name=fig:results_chip:legend,
			\common,
		}
		\caption{%
			$\lambda = 10^{-4}$.
		}%
		\label{fig:results_chip_lambda0.0001}
	\end{subfigure}% important
	\hfill
	\ref*{fig:results_chip:legend}% chktex 2
	\hspace*{\fill}
	\caption{%
		Results of the chip example with varying \(H = \lambda I_p\).
		\rreplotcommoncaption}%
	\label{fig:results_chip}
\end{figure}

\subsection{Example 3: steel rail profile}\label{sct:rail}
This example concerns the heat transfer in a steel rail profile~\cite{morwiki_steel}.
It considers a 2D structural cross-section of the rail semi-discretized using \ac{FEM} with $d = \num{317,377}$.
There are $p = 7$ inputs which parameterize the temperature on the boundary of the steel rail.
In the control problem associated with this benchmark, these inputs represent the temperature of the cooling fluid used to control the internal temperature of the steel rail.
Related to these inputs, the $q = 6$ outputs are chosen in line with the control objective of minimizing temperature differences at specific points in the geometry.
See~\cite{BenS05b} for more details.

For this example,
we consider the \ac{ARE} with $H = 10^{-4} \cdot I_p$ and, in addition, two Lyapunov equations
(which are linear special cases of the \ac{ARE}).
First, the Lyapunov equation associated with the controllability Gramian $P \in \mathbb{R}^{d \times d}$:
\begin{equation}
	A PE^\TT + EP A^\TT + B B^\TT = 0.
	\label{eqn:lyap_contr}
\end{equation}
Second, the Lyapunov equation associated with the observability Gramian $E^\TT QE \in \mathbb{R}^{d \times d}$:
\begin{equation}
	A^\TT QE + E^\TT Q A + C^\TT C = 0.
	\label{eqn:lyap_obsv}
\end{equation}

The results for the \ac{ARE} and two Lyapunov equations are shown in Figure~\ref{fig:results_rail}.
Since in this example $A$ and $E$ are symmetric (with $E \neq I$ positive definite), the shift parameters are real-valued.
For the Lyapunov equations, the \ac{RRE} acceleration does not significantly reduce the residual, indicating that the iterates generated by ADI are already nearly optimal (in the sense of optimization problem~\eqref{eqn:matrre_opt}).
For the \ac{ARE}, on the other hand, the \ac{RRE} acceleration significantly reduces the residual compared to the non-accelerated RADI in iterations where stagnation occurs.
In this way, the number of iterations required to reach the convergence criterion (the dashed horizontal line in Figure~\ref{fig:results_rail_riccati}) is reduced from 107 to 93.
In cycling mode, the restart in iteration 66 for the \ac{ARE} again leads to stagnation; see Remark~\ref{remark:cycling_mode_stagnation}.

\begin{figure}[t]
	\centering
	\newcommand\common{
		xmax = 110,
		ymin = 5e-15,
		ytickten = {-15,-12,...,0}, % must contain 0; chktex 11
	}
	\begin{subfigure}{0.5\textwidth}
		\centering
		\rreplot{steel_profile_lyap-contr}{
			legend entries={(R)ADI, (R)ADI+RRE (non-cycling), (R)ADI+RRE (cycling)}, % chktex 36
			legend to name=fig:results_rail:legend,
			\common,
		}
		\caption{\ac{ADI} for Controllability Lyapunov eq.~\eqref{eqn:lyap_contr}.}%
        \label{fig:results_rail_lyap_contr}
	\end{subfigure}%
	\hfill
	\begin{subfigure}{0.5\textwidth}
		\centering
		\rreplot{steel_profile_lyap-obsv}{\common}
		\caption{\ac{ADI} for Observability Lyapunov eq.~\eqref{eqn:lyap_obsv}.}%
        \label{fig:results_rail_lyap_obsv}
	\end{subfigure}%
	\\
	\bigskip
	\begin{subfigure}[c]{0.5\textwidth}
		\centering
		\rreplot{steel_profile_riccati}{\common}
		\caption{RADI for \ac{ARE}~\eqref{eqn:riccati}.}%
        \label{fig:results_rail_riccati}
	\end{subfigure}%
	\hfill
	\ref*{fig:results_rail:legend}% chktex 2
	\hspace*{\fill}
	\caption{%
		Results of the steel rail profile example.
		\rreplotcommoncaption}%
    \label{fig:results_rail}
\end{figure}

\subsection{Example 4: triple chain}~\label{sct:triplechain}
In this section, we consider an example of a mechanical system.
The example is based on~\cite[Example~2]{truhar2009efficient} and consists of a configuration of interconnected masses, springs, and dampers.
In second-order formulation, the model is given by
\begin{equation}
\begin{aligned}
	M_\mathrm{mass} \ddot{x} + C_\mathrm{damp} \dot{x} + K_\mathrm{stiff} x & = B_u u, \\
		y & = C_y x.
\end{aligned}
\label{eqn:ss_secondorder}
\end{equation}
We use \num{30,001} masses, giving $d = \num{60,002}$ ($p = q = 1$) in the following first-order formulation~\eqref{eqn:ss}:
\begin{align*}
	E = \begin{bmatrix} -K_\mathrm{stiff} & 0 \\ 0 & M_\mathrm{mass} \end{bmatrix}, \quad & A = \begin{bmatrix} 0 & -K_\mathrm{stiff} \\ -K_\mathrm{stiff} & -C_\mathrm{damp} \end{bmatrix}, \\
	C = \begin{bmatrix} C_y & 0 \end{bmatrix}, \quad & B = \begin{bmatrix} 0 & B_u^\TT \end{bmatrix}^\TT.
\end{align*}
Next, we apply a perfect shuffle permutation such that $A$ and $E$ have almost banded structure~\cite{kressner2019lowrank}, which we observed to be able to significantly improve numerical conditioning.
The Rayleigh damping factors $\alpha$ and $\beta$ which relate the damping matrix $C_\mathrm{damp}$ to the mass matrix $M_\mathrm{mass}$ and the stiffness matrix $K_\mathrm{stiff}$ as in
\begin{equation}
	C_\mathrm{damp} = \alpha M_\mathrm{mass} + \beta K_\mathrm{stiff}
\end{equation}
are both set to 0.1.
The viscosity of the dampers is set to $\nu = 5$ and for the masses and stiffness coefficients we use the following values (see~\cite[Figure~5.1]{truhar2009efficient}):
\begin{equation}
\begin{aligned}
	m_0 & = 10, & m_1 & = 1, & m_2  & = 2, & m_3  & = 3, \\
	k_0 & = 50, & k_1 & = 10, & k_2 & = 20, & k_3 & = 1.
\end{aligned}
\end{equation}

For this example, we evaluate the method using the \ac{ARE}~\eqref{eqn:riccati} and several settings for the window size $n$.
The results are displayed in Figure~\ref{fig:results_triplechain}.
The residual of the extrapolants is (significantly) higher than the residual of the RADI iterates for all window sizes~$n$ for the first 30 iterations.
After these initial 30 iterations,
in the non-cycling mode,
we would expect the \ac{RRE} residual to decrease faster for larger~$n$,
while for $n=6$ and $n=12$ the \ac{RRE} residual sometimes even exceed the RADI residuals up to iteration 50.
All these observations are likely caused by the nonlinearity of the \ac{ARE}~\eqref{eqn:riccati}.
Starting at iteration 56, during the last stagnation phase before the termination of RADI,
the picture changes: the wider $n=12$ yields a small but consistent acceleration,
while the \ac{RRE} residuals of the narrower~$n$ collapse onto the RADI residuals.
This may, however, be caused by the (default) optimizer of MATLAB we used and its associated numerical tolerances.
Enabling cycling mode in this example again leads to stagnation of the residual; see Remark~\ref{remark:cycling_mode_stagnation}.
Possibly, restarting before the iterate is sufficiently close to the solution exacerbates the numerical ill-conditioning of the subsequent RADI iterates.

% Our RRE implementation counts an ADI double step as one.
% That is, even though two complex shifts have been processed,
% only the second iterate counts towards the limit $n$,
% since the true first iterate (after only one of the two complex shifts) is not available.
% In the case $n=12$, after the restart at iteration 68,
% this counting of pairs of shifts "delays" the next extrapolate
% until iteration 85 (instead of 80).
% To show this phenomenon in the plot, increase xmax to 82.
\begin{figure}[t]
	\centering
	\newcommand\common{
		xmax = 82,
		ymax = 1e7,
		ymin = 1e-12,
		ytickten = {-12,-8,...,8}, % must contain 0; chktex 11
	}
	\begin{subfigure}{0.5\textwidth}
		\centering
		\rreplot{triple_chain_M3}{
			legend entries={RADI, RADI+RRE (non-cycling), RADI+RRE (cycling)},
			legend to name=fig:results_triplechain:legend,
			\common,
		}
		\caption{$n = 3$.}%
        \label{fig:results_triplechain_M3}
	\end{subfigure}%
	\hfill
	\begin{subfigure}{0.5\textwidth}
		\centering
		\rreplot{triple_chain_M6}{\common}
		\caption{$n = 6$.}%
        \label{fig:results_triplechain_M5}
	\end{subfigure}%
	\\
	\bigskip
	\begin{subfigure}[c]{0.5\textwidth}
		\centering
		\rreplot{triple_chain_M12}{\common}
		\caption{$n = 12$.}%
        \label{fig:results_triplechain_M10}
	\end{subfigure}%
	\hfill
	\ref*{fig:results_triplechain:legend}% chktex 2
	\hspace*{\fill}
	\caption{%
		Results of the triple chain example with varying \ac{RRE} window size \(n\).
		\rreplotcommoncaption}%
    \label{fig:results_triplechain}
\end{figure}

\subsection{Error of RRE solution on small versions of the examples}
An \ac{ARE} can have many solutions, while typically one is only interested in the unique stabilizing solution.
Under some technical assumptions, as discussed in Section~\ref{sct:application}, RADI is guaranteed to converge to this stabilizing solution.
As mentioned in Remark~\ref{remark:convergence_psd_solution}, no such guarantees exist for the \ac{RRE} solution.

Hence, we confirm that the \ac{RRE} solution is close to the desired \ac{PSD} solution on small-scale versions of all numerical examples presented in this section (except the chip example presented in Section~\ref{sct:chip}).
We compute the relative error of the \ac{RRE} solution $\widehat{X}$ with respect to the exact solution $X$ (computed using the MATLAB function \texttt{icare}) given by $\norm{X}^{-1} \cdot \norm{ X - \widehat{X} }$.
The results are as follows:
\begin{itemize}
	\item Section~\ref{sct:toeplitz} (Toeplitz system matrix): size $d = 500$ and relative error $5 \cdot 10^{-14}$.
	\item Section~\ref{sct:rail} (steel rail profile): size $d = 371$ and relative error $7 \cdot 10^{-13}$.
	\item Section~\ref{sct:triplechain} (triple chain): size $d = 302$ and relative error $4 \cdot 10^{-9}$.
\end{itemize}

\section{Conclusion and outlook}\label{sct:conclusion}
We proposed an extension of \ac{RRE} that allows it to be applied to nonstationary fixed-point iterations and problems involving large-scale low-rank matrices.
We formulated an efficient least-squares problem to find the \ac{RRE} extrapolant.
Although our scheme can be applied to arbitrary sequences of low-rank matrices, we specifically discussed its application to accelerating iterative solvers of large-scale matrix equations (in particular the \ac{ARE}).
We analyzed some theoretical properties of the extrapolation applied to the RADI algorithm and showed that linear equality constraints on the \ac{RRE} coefficients are sufficient to prevent convergence to non-\ac{PSD} solutions for a non-restarting \ac{RRE} scheme.
However, we also conjecture that for a restarting \ac{RRE} scheme additional constraints are required which diminish the acceleration potential for \acp{ARE}.
The numerical examples that we treated illustrated the potential of the approach.
We specifically investigated how the convergence acceleration may be affected by the problem size, the type of equation and the number of iterates used for extrapolation.

In the future, we intend to investigate more comprehensively the benefits of \ac{RRE} in terms of computational time, and this requires developing a systematic strategy to adaptively enables \ac{RRE} in the iteration.
\chadded{%
A study of \ac{RRE} acceleration for multi-term Sylvester equations has already been investigated in separate work~\cite{RREmultiterm26}, which demonstrates substantial improvements in both convergence speed and execution time.}
Another interesting problem is to derive convergence guarantees for our \ac{RRE} method applied to \acp{ARE}.
As mentioned in Remark~\ref{remark:convergence_psd_solution}, convergence of \ac{RRE} for general nonlinear systems is a difficult task if we only have knowledge of some regularity properties of the nonlinearity.
For the specific case of \acp{ARE}, we have an explicit form of the nonlinearity (formula~\eqref{eqn:Xi_fp_nonstationary_matrix} and Algorithm~\ref{alg:radi}) which may allow us to obtain convergence results that, in contrast to existing results in literature, do not require heuristic assumptions that are hard to verify.

\appendix
% chktex-file 3 % disable {} warnings for exponents etc.

\section{Runtime and late RRE computation}%
\label{append:runtime}

In the main part of the manuscript we refrained from reporting execution times,
since applying \ac{RRE} at every possible step is, of course, excessive.
A potential remedy is to only compute extrapolates once RADI is sufficiently close to the desired tolerance.
The quantification of this closeness is clearly problem dependent, and developing either a rigorous characterization or an effective heuristic is  beyond the scope of this manuscript.

As Figures~\ref{fig:results_toeplitz} and~\ref{fig:results_rail} show,
\ac{RRE} is most effective when the underlying process plateaus just before the target tolerance.
Therefore, one may chose some residual thresholds below which extrapolation is enabled,
\unskip\footnote{%
	See documentation of \texttt{mess\_lrradi}; specifically
	\texttt{opts.radi.rre.enable\_tol} and
	\texttt{opts.radi.rre.check\_res\_tol}.
}
determined a posteriori from Figures~\ref{fig:results_toeplitz} through~\ref{fig:results_rail}.
Table~\ref{tab:runtime} shows the resulting runtimes of RADI without \ac{RRE} and of
RADI with late non-cycling \ac{RRE}
for the promising candidates among all examples.
We omit the examples for which \ac{RRE} does not reduce the number of RADI iterations required to reach the prescribed residual reduction of $10^{-10}$,
i.e., where one can not expect a runtime advantage using \ac{RRE}.
In the shown subset of experiments,
\ac{RRE} reduced the total runtime by up to \qty{20}{\percent},
while increasing runtime only for a single experiment by \qty{1}{\percent}.

\begin{table}[t]
	\centering
	\begin{tabular}{%
			@{} % trim rules; recommended by booktabs
			l % underlying example
			S[table-format=3] % number of RADI iterations
			S[table-format=3.2] % runtime RADI
			S[table-format=1e-2] % threshold
			c % number of RRE extrapolates
			S[table-format=2] % number of RADI iterations
			S[table-format=2.2] % runtime RADI+RRE
			@{} % trim rules; recommended by booktabs
		}
		\toprule
		& \multicolumn{2}{c}{{RADI}} & \multicolumn{4}{c}{{RADI+RRE}} \\
		\cmidrule(rl){2-3}
		\cmidrule(l){4-7}
		Example & {\#RADI} & {time [s]} & {Threshold} & {\#RRE} & {\#RADI} & {time [s]} \\
		\midrule
		Figure~\ref{fig:results_toeplitz_q1} & 44 & 1.67 & 1e-9 & 1 & 30 & 1.67 \\
		Figure~\ref{fig:results_toeplitz_q40} & 76 & 77.12 & 1e-9 & 1 & 65 & 66.27 \\
		\addlinespace
		Figure~\ref{fig:results_toeplitz_q20_linear} & 23 & 6.16 & 1e-9 & 1 & 20 & 4.9 \\
		Figure~\ref{fig:results_toeplitz_q40_linear} & 22 & 10.26 & 1e-9 & 1 & 20 & 10.4 \\
		\addlinespace
		Figure~\ref{fig:results_rail_riccati} & 107 & 171.68 & 5e-10 & 5 & 93 & 154.67 \\
		\bottomrule
	\end{tabular}
	\caption{%
		Runtimes of RADI versus RADI plus late \ac{RRE}.
		The columns designate
		the underlying example,
		the number of RADI iterations and total runtime,
		the threshold below which \ac{RRE} becomes active,
		number of \ac{RRE} extrapolates computed,
		the number of RADI iterations and total runtime,
		respectively.
		All timings are in seconds.
	}\label{tab:runtime}
\end{table}

\section{Proof of Theorem~\ref{thm:residual-expression}}
\label{append:proof-residual-expression}

\begin{proof}
%\todo{while I think having the notation close to where it is used in the proof is good, I also think this information is too important to be ``hidden'' in the appendix.}

\newcommand\KK[1]{\chadded{K_{#1}}}
\newcommand\KT[1]{\chadded{K_{#1}^\TT}}
\newcommand\DeltaKK[1]{\chadded{\Delta K_{#1}}}
\newcommand\DeltaKT[1]{\chadded{\Delta K_{#1}^\TT}}

%In addition, we define the \ac{RRE} weights $\rr_i\in\mathbb{R}$, $i=1\colon k$.
\chadded{To simplify notation, we define}
\begin{equation*}
	\KK{k} := B^\TT X_k E
	\quad\text{and}\quad
	\DeltaKK{k} :=
	B^\TT (X_k - X_{k-1}) E =
	B^\TT V_k\Y_k^{-1}V_k^\TT E.
\end{equation*}
Note that Algorithm~\ref{alg:nonstationary_lowrank_rre} is essentially Algorithm~\ref{alg:radi} equipped with \ac{RRE}.
Before starting the proof we need the following auxiliary identities, which can
be shown, after some manipulation, from the modified RADI iteration
(Algorithm~\ref{alg:radi}), that is,
\begin{multline}\label{eq:AVkYinvVkT}
	A^\TT V_k\Y_k^{-1}V_k^\TT E
	= R_{k-1}T(R_k^\TT-R_{k-1}^\TT)
	+ \KT{k-1} H^{-1} \DeltaKK{k}
	- \sigma_{k-1} E^\TT V_k\Y_k^{-1}V_k^\TT E
\end{multline}
and
\begin{align}\label{eq:long-symterm}
	-\rr_k^2 \DeltaKT{k} H^{-1} \DeltaKK{k}
	&= -\rr_k^2 E^\TT V_k\Y_k^{-1}\cdot 2\sigma_{k-1}(T-\Y_k) \cdot \Y_k^{-1}V^\TT_k E \notag\\
	&= \rr_k^2(R_k-R_{k-1})T(R_k^\TT-R_{k-1}^\TT) +
	2\rr_k^2\sigma_{k-1}E^\TT V_k\Y_k^{-1}V_k^\TT E,
\end{align}
where
\begin{equation}\label{eq:Re-symterm}
	2\sigma_k E^\TT V_k\Y_k^{-1}V_k^\TT E =-(R_k-R_{k-1})\Y_k(R_k^\TT-R_{k-1}^\TT).
\end{equation}

The proof is by induction on the number $n$ of extrapolation terms. First consider the case $n=2$. From the
assumption $\rr_1+\rr_2 = 1$ and the equation (with $V_1\equiv Z_1$ and $\Y_1^{-1}\equiv D_1$)
\begin{equation}\label{eq:X_k}
	X_k=\sum_{i=1}^{k}V_i\Y_i^{-1}V_i^\TT, \quad k=1,2,\dots,
\end{equation}
we have $\R(\rr_1X_1 + \rr_2X_2) =	\R(X_1 + \rr_2 V_2\Y_2^{-1}V_2^\TT)$. Using~\eqref{eqn:riccati},
%the identity for the residual matrix
%\[
%\R(\Xi) = A^\TT\Xi E + E^{\TT}\Xi A + C^\TT C - E^\TT \Xi B H^{-1} B^\TT\Xi E,
%\]
the expression then can be expanded as
\begin{multline*}
	\R(\rr_1X_1 + \rr_2X_2) = \R(X_1) + \rr_2(A^\TT V_2\Y_2^{-1}V_2^\TT E + \square^\TT) \\
	-(\rr_1+\rr_2)\rr_2( \KT{1} H^{-1} \DeltaKK{2} + \square^\TT)
	- \rr_2^2 \DeltaKT{2} H^{-1} \DeltaKK{2},
\end{multline*}
where and henceforth we write the ``$\square^\TT$'' symbol rightmost inside a bracket as a shorthand for the transpose of the preceding terms (for the sake of brevity).
It follows from~\eqref{eq:AVkYinvVkT} and~\eqref{eq:long-symterm} and then~\eqref{eq:Re-symterm} that
\begin{align*}
	\MoveEqLeft[3]
	\R(\rr_1X_1 + \rr_2X_2)
	\\
	={}& \R(X_1) + \rr_2\left(R_{1}T(R_2^\TT-R_{1}^\TT)
		+\square^\TT\right) - 2\rr_2\sigma_1 E^\TT V_2\Y_2^{-1}V_2^\TT E \\
	& + \big((\rr_2-(\rr_1+\rr_2)\rr_2)
		\KT{1} H^{-1} \DeltaKK{2} +\square^\TT\big) \\
	& + \rr_2^2(R_2-R_{1})T(R_2^\TT-R_{1}^\TT) +
		2\rr_2^2\sigma_1E ^\TT V_2\Y_2^{-1}V_2^\TT E
	\\
	={}& \R(X_1) + \rr_2\left(R_{1}T(R_2^\TT-R_{1}^\TT) +\square^\TT\right)
	+ \rr_2^2(R_2-R_{1})T(R_2^\TT-R_{1}^\TT) \\
	& + 2(\rr_2^2-\rr_2)\sigma_1E^\TT V_2\Y_2^{-1}V_2^\TT E
	\\
	={}& \R(X_1) + \rr_2\left(R_{1}T(R_2^\TT-R_{1}^\TT) +\square^\TT\right)
	+ \rr_2^2(R_2-R_{1})T(R_2^\TT-R_{1}^\TT) \\
	& + (\rr_2-\rr_2^2)(R_2-R_{1})\Y_2(R_2^\TT-R_{1}^\TT).
\end{align*}
The final expression can equivalently be written as the factorization
\begin{equation*}%\label{eq:range-fac-2}
	\R(\rr_1X_1 + \rr_2X_2) =
	[ R_1 \ R_2 ]
	\mathscr{H}_2
	[ R_1 \ R_2 ]^\TT
	,\quad
	\mathscr{H}_2\equiv  \mathscr{H}_2(\rr_1,\rr_2) = \mathscr{A}_2 + \mathscr{B}_2,
\end{equation*}
where
\begin{equation*}
	\mathscr{A}_2 = \blkdiag(T, \textbf{0}_{p\times p}),\quad
	\mathscr{B}_2 = \begin{bmatrix}
		(\rr_2^2-2\rr_2)T+(\rr_2-\rr_2^2)\Y_2 & (\rr_2-\rr_2^2)(T-\Y_2) \\
		(\rr_2-\rr_2^2)(T-\Y_2)    &  \rr_2^2 T +(\rr_2-\rr_2^2)\Y_2
	\end{bmatrix}.
\end{equation*}
This clearly shows $\range (\R(\rr_1X_1 + \rr_2X_2))\subseteq \range([R_1\ R_2])$.
Although in the expression for $\mathscr{H}_2$ there is only one free parameter $\rr_2$ as a result of the constraint $\rr_1+\rr_2=1$, here we still write out explicitly the dependence on all the parameters as $\mathscr{H}_2(\rr_1,\rr_2)$; same for all the
$\mathscr{H}_k$ below.

We will see later that the case $n=2$ is special in that the terms
$ \KT{1} H^{-1} \DeltaKK{2} $ and its transpose disappear because their
coefficients cancel out as $\rr_2-(\rr_1+\rr_2)\rr_2=0$, which is a property that does not hold for
arbitrary $n\neq 2$.

In order to see the pattern, we also present the case of $n=3$ in detail (as the \emph{base case}
for induction).
We have, using the assumption $\rr_1+\rr_2+\rr_3=1$ and~\eqref{eq:AVkYinvVkT}--\eqref{eq:X_k},
\begin{align*}
	\MoveEqLeft[3]
	\R(\rr_1X_1 + \rr_2X_2 + \rr_3X_3)
	\\
	={}& \R(\rr_1X_1 + (\rr_2+\rr_3)X_2 +
	\rr_3V_3\Y_3^{-1}V_3^\TT)
	\\
	={}& \R(\rr_1X_1 + (\rr_2+\rr_3)X_2) + \rr_3(A^\TT V_3\Y_3^{-1}V_3^\TT E + \square^\TT)\\
	&- \rr_3^2 \DeltaKT{3} H^{-1} \DeltaKK{3} \\
	&-\rr_3\big(E^\TT(\rr_1 X_1+(\rr_2+\rr_3)X_2)B H^{-1} \DeltaKK{3} +\square^\TT\big)
	\\
	={}& \R(\rr_1X_1 + (\rr_2+\rr_3)X_2) + \rr_3\left(R_{2}T(R_3^\TT-R_{2}^\TT) +\square^\TT\right)
	- 2\rr_3\sigma_2 E^\TT V_3\Y_3^{-1}V_3^\TT E\\
	&+ \rr_3\big(E^\TT((1-\rr_2-\rr_3)X_2-\rr_1X_1)B H^{-1} \DeltaKK{3}
	+\square^\TT\big) \\
	&+ \rr_3^2(R_3-R_{2})T(R_3^\TT-R_{2}^\TT) +
	2\rr_3^2\sigma_2E^\TT V_3\Y_3^{-1}V_3^\TT E
	\\
	={}& \R(\rr_1X_1 + (\rr_2+\rr_3)X_2) + \rr_3\left(R_{2}T(R_3^\TT-R_{2}^\TT) +\square^\TT\right) \\
	&+ \rr_3^2(R_3-R_{2})T(R_3^\TT-R_{2}^\TT) \\
	&+ (\rr_3-\rr_3^2)(R_3-R_{2})\Y_3(R_3^\TT-R_{2}^\TT)\\
	&+ \rr_3\big(E^\TT\rr_1(X_2-X_1)B H^{-1} \DeltaKK{3} +\square^\TT\big),
\end{align*}
where $\range (\R(\rr_1X_1 + (\rr_2+\rr_3)X_2))\subseteq \range([R_1\ R_2])$ from the case of $n=2$, and so we only need to investigate the last two terms:
\begin{equation*}%\label{eq:cross-term-2}
	E^\TT(X_2-X_1)B H^{-1} \DeltaKK{3}  +\square^\TT
	= \DeltaKT{2} H^{-1} \DeltaKK{3}
	+ \DeltaKT{3} H^{-1} \DeltaKK{2}
\end{equation*}
by definition of $\DeltaKK{2}$.
Defining
\begin{equation}\label{eq:alpha-Cik}
	\alpha_{ik}=
	\frac{1}{2\sqrt{\sigma_{i-1}\sigma_{k-1}}}\in\mathbb{R},\qquad
	C_{ik} = V_i^\TT B H^{-1}B^\TT V_k \in\mathbb{R}^{p\times p},
\end{equation}
we have
\begin{equation*}
	\DeltaKT{2} H^{-1} \DeltaKK{3} + \square^\TT
	=  \alpha_{23}  ((R_2-R_1)  C_{23}  (R_3^\TT - R_2^\TT)
	    + (R_3 - R_2)C_{23}^\TT(R_2^\TT-R_1^\TT) ),
\end{equation*}
and, therefore,
\begin{equation*}
	\rr_3 (\rr_1 \DeltaKT{2} H^{-1} \DeltaKK{3} +\square^\TT) =
	[R_1\ R_2\ R_3]\mathscr{C}_3 [R_1\ R_2\ R_3]^\TT,
\end{equation*}
where
\begin{equation*}
	\mathscr{C}_3 = \rr_3\rr_1 \alpha_{23}
	\begin{bmatrix}
		\textbf{0} & C_{23} & -C_{23} \\
		C_{23}^\TT  & -C_{23} - C_{23}^\TT & C_{23} \\
		-C_{23}^\TT &  C_{23}^\TT & \textbf{0}
	\end{bmatrix}.
\end{equation*}
Then, the factorization follows:
\begin{align*}%\label{eq:range-fac-3}
	\R(\rr_1X_1 + \rr_2X_2 + \rr_3X_3) =
	[R_1 \ R_2 \ R_3]
	\mathscr{H}_3
	[R_1 \ R_2 \ R_3]^\TT,
\end{align*}
where
\begin{align*}
  \mathscr{H}_3 \equiv  \mathscr{H}_3(\rr_1,\rr_2,\rr_3) = \mathscr{A}_3 + \mathscr{B}_3 +
	\mathscr{C}_3,
\end{align*}
with
\begin{equation*}
	\mathscr{A}_3 = \blkdiag(\mathscr{H}_2(\rr_1,\rr_2+\rr_3), \textbf{0}_{p\times p}),
\end{equation*}
and
\begin{equation*}
	\mathscr{B}_3 = \blkdiag\left(\textbf{0}_{p\times p},
	\begin{bmatrix}
		(\rr_3^2-2\rr_3)T+(\rr_3-\rr_3^2)\Y_3    &   (\rr_3-\rr_3^2)(T-\Y_3) \\
		(\rr_3-\rr_3^2)(T-\Y_3) & \rr_3^2 T +(\rr_3-\rr_3^2)\Y_3
	\end{bmatrix}\right).
\end{equation*}
This clearly shows $\range (\R(\rr_1X_1 + \rr_2X_2 + \rr_3X_3))
\subseteq \range([R_1\ R_2\ R_3])$.

Assume as the \textit{induction hypothesis} ($n=k-1$)
$\range (\R(\sum_{i=1}^{k-1}\zeta_i X_i))\subseteq \range([R_1\ R_2\ \cdots\
R_{k-1}])$
for any set of real numbers $\{\zeta_1,\zeta_2,\dots,\zeta_{k-1}\}$ such that  $\sum_{i=1}^{k-1}\zeta_i=1$, and, specifically,
\begin{equation*}
	\R\left(\sum_{i=1}^{k-1}\zeta_i X_i\right) =
	[	R_1 \ R_2 \ \cdots \ R_{k-1}]
	\mathscr{H}_{k-1}(\zeta_1,\zeta_2,\dots,\zeta_{k-1})
	[	R_1 \ R_2 \ \cdots \ R_{k-1}]^\TT
\end{equation*}
holds for some symmetric matrix
$\mathscr{H}_{k-1}(\zeta_1,\zeta_2,\dots,\zeta_{k-1})\in\mathbb{R}^{(k-1)p\times (k-1)p}$.
We now prove
\begin{equation*}
  \range \left(\R\left(\sum_{i=1}^{k}\rr_i X_i\right)\right)\subseteq
  \range([R_1\ R_2\ \cdots\ R_{k}])
\end{equation*}
with $\sum_{i=1}^{k}\rr_i=1$ and show how to efficiently obtain the factorization of the residual
of the $k$-term extrapolant. Similarly, by using the assumption $\sum_{i=1}^{k}\rr_i=1$
and~\eqref{eq:AVkYinvVkT}--\eqref{eq:X_k}, we have
\begin{align*}
	\MoveEqLeft[3]
	\R(\rr_1 X_1 + \rr_2 X_2 + \cdots+ \rr_k X_k) \\
	={}& \textstyle \R\big(\sum_{j=1}^{k-2}\rr_{j}X_j
	+ (\rr_{k-1}+\rr_k)X_{k-1} + \rr_k V_k Y_k^{-1}V_k^\TT \big)
	\\
	={}& \textstyle \R\big(\sum_{j=1}^{k-2}\rr_{j}X_j + (\rr_{k-1}+\rr_k)X_{k-1}\big)
	+ \rr_k(A^\TT V_k\Y_k^{-1}V_k^\TT E + \square^\TT) \\
	&- \rr_k^2 \DeltaKT{k} H^{-1} \DeltaKK{k} \\
	& \textstyle -\rr_k\big(E^\TT(\sum_{j=1}^{k-2}\rr_{j}X_j +
	(\rr_{k-1}+\rr_k)X_{k-1}) B H^{-1} \DeltaKK{k} +\square^\TT\big)
	\\
	={}& \textstyle \R\big(\sum_{j=1}^{k-2}\rr_{j}X_j + (\rr_{k-1}+\rr_k)X_{k-1}\big)\\
    &+\rr_k\big(R_{k-1}T(R_k^\TT-R_{k-1}^\TT) +\square^\TT\big) -
	  2\rr_k\sigma_{k-1} E^\TT V_k\Y_k^{-1}V_k^\TT E \\
	&+ \textstyle \rr_k\big(E^\TT((1-\rr_{k-1}-\rr_k)X_{k-1}-\sum_{j=1}^{k-2}\rr_{j}X_j)
	B H^{-1} \DeltaKK{k} + \square^\TT\big) \\
	&+  \textstyle  \rr_k^2(R_k-R_{k-1})T(R_k^\TT-R_{k-1}^\TT) +
	2\rr_k^2\sigma_{k-1} E^\TT V_k\Y_k^{-1}V_k^\TT E
	\\
	={}&\textstyle \R\big(\sum_{j=1}^{k-2}\rr_{j}X_j + (\rr_{k-1}+\rr_k)X_{k-1}\big)
	+ \rr_k(R_{k-1}T(R_k^\TT-R_{k-1}^\TT) +\square^\TT)\\
	&{} + \rr_k^2(R_k-R_{k-1})T(R_k^\TT-R_{k-1}^\TT) +
		(\rr_k-\rr_k^2)(R_k-R_{k-1})\Y_k(R_k^\TT-R_{k-1}^\TT) \\
	&\textstyle  + \rr_k\big(E^\TT\sum_{j=1}^{k-2}\rr_{j}(X_{k-1}-X_j)B H^{-1} \DeltaKK{k} +\square^\TT\big).
\end{align*}
We have, using~\eqref{eq:X_k} again,
\begin{align*}
	\MoveEqLeft[3]
	\rr_k\sum_{j=1}^{k-2}\rr_{j}(X_{k-1}-X_j) \\
	={}& \rr_k\sum_{j=1}^{k-2}\rr_{j}\big(
	\sum_{i=1}^{k-1}V_i\Y_i^{-1}V_i^\TT - \sum_{i=1}^{j}V_i\Y_i^{-1}V_i^\TT \big)
	\\
	={}& \textstyle \rr_k\sum_{j=1}^{k-2}\rr_{j} \sum_{i=j+1}^{k-1}V_i\Y_i^{-1}V_i^\TT
	\\
	={}& \textstyle \rr_k\rr_1V_2\Y_2^{-1}V_2^\TT + \rr_k(\rr_1+\rr_2)V_3\Y_3^{-1}V_3^\TT + \cdots\\
	&+ \rr_k(\rr_1+\rr_2+\cdots+\rr_{k-2}) V_{k-1}\Y_{k-1}^{-1}V_{k-1}^\TT \\
	=:& \textstyle \sum_{i=2}^{k-1} \beta_{ki}V_{i}\Y_{i}^{-1}V_{i}^\TT,
\end{align*}
where we have defined $\beta_{ki}=\rr_k\sum_{j=1}^{i-1}\rr_j$.
Therefore, we have
\begin{equation*}
	\DeltaKT{i} H^{-1} \DeltaKK{k}
	=\alpha_{ik}(R_i-R_{i-1})C_{ik}(R_k^\TT-R_{k-1}^\TT),
\end{equation*}
using the definitions from~\eqref{eq:alpha-Cik}. After some manipulation,
\begin{equation*}
	\rr_k\left(\sum_{j=1}^{k-2}\rr_{j}E^\TT(X_{k-1}-X_j) B H^{-1} \DeltaKK{k}
	+\square^\TT\right) =
	[R_1\ R_2\ \cdots\ R_k]\mathscr{C}_k [R_1\ R_2\ \cdots\ R_k]^\TT,
\end{equation*}
where
\begin{equation*}
	\mathscr{C}_k = \widetilde{\mathscr{C}}_k + \widetilde{\mathscr{C}}_k^\TT,\quad
	\widetilde{\mathscr{C}}_k =
	\begin{bNiceArray}[
		first-row,code-for-first-row=\scriptstyle,
		last-col,code-for-last-col=\scriptstyle,
		]{ccc}
		(k-2)p       &  p           &  p            &  \\
		\textbf{0}   & M_k^d-M_k^t          & M_k^t-M_k^d      & (k-1)p \\
		\textbf{0}   & \textbf{0}   & \textbf{0}    & p
	\end{bNiceArray}\; , \quad k\ge 3,
\end{equation*}
and
\begin{align*}
	M_k &= [
	\beta_{k2}\alpha_{2k}C_{2k}^\TT \;\;
	\beta_{k3}\alpha_{3k}C_{3k}^\TT \;\;
	\dots \;\;
	\beta_{k,k-1}\alpha_{k-1,k}C_{k-1,k}^\TT
	]^\TT,
	\\
	M_k^d &= [ M_k^\TT \ \textbf{0} ]^\TT,
	\\
	M_k^t &= [ \textbf{0} \ M_k^\TT ]^\TT.
\end{align*}
Then, the factorization follows:
\begin{equation*}%\label{eq:range-fac-k}
	\R\left(\sum_{i=1}^{k}\rr_i X_i\right)
	=
	[R_1 \ R_2 \ \cdots \ R_k]
	\mathscr{H}_k
	[R_1 \ R_2 \ \cdots \ R_k]^\TT,
\end{equation*}
where
\begin{equation*}
	\mathscr{H}_k
	\equiv  \mathscr{H}_k(\rr_1,\rr_2,\dots,\rr_k) = \mathscr{A}_k + \mathscr{B}_k + \mathscr{C}_k
\end{equation*}
and
\begin{equation*}
	\mathscr{A}_k = \blkdiag
	\left(\mathscr{H}_{k-1}(\rr_1,\rr_2,\dots,\rr_{k-2},\rr_{k-1}+\rr_{k}), \textbf{0}_{p\times
	p}\right).
\end{equation*}
Here, the existence of $\mathscr{H}_{k-1}(\rr_1,\rr_2,\dots,\rr_{k-2},\rr_{k-1}+\rr_{k})\equiv
\mathscr{H}_{k-1}$ such that
\[
\R\left(\sum_{j=1}^{k-2}\rr_{j}X_j + (\rr_{k-1}+\rr_k)X_{k-1}\right)= [R_1 \ R_2 \ \cdots \ R_{k-1}]
\mathscr{H}_{k-1}
[R_1 \ R_2 \ \cdots \ R_{k-1}]^\TT,
\]
follows from the inductive hypothesis, and
\begin{equation*}
	\mathscr{B}_k = \blkdiag\left(\textbf{0}_{(k-2)p\times (k-2)p},
	\begin{bmatrix}
		(\rr_k^2-2\rr_k)T+(\rr_k-\rr_k^2)\Y_k    &   (\rr_k-\rr_k^2)(T-\Y_k) \\
		(\rr_k-\rr_k^2)(T-\Y_k) & \rr_k^2 T +(\rr_k-\rr_k^2)\Y_k
	\end{bmatrix}\right).
\end{equation*}
The range property $\range (\R(\sum_{i=1}^{k}\rr_i X_i))\subseteq \range([R_1\ R_2\ \cdots\ R_{k}])$
follows (for $n=k$) and hence the induction is completed.
\end{proof}

\section{Proof of Theorem~\ref{thm:residual-psd-cond}}\label{append:proof-residual-psd-cond}
\begin{proof}
Clearly, if we can factorize $\mathscr{H}_2$ into the form of
$HH^\TT$ for some $H$, then a factorization of the same form is immediately available for $\R(\rr_1X_1 + \rr_2X_2)$, which is equivalent to say the matrix is positive semi-definite.

If $T$ is positive semi-definite, $\Y_2$ is also positive semi-definite from its construction by the RADI algorithm; see Algorithm~\ref{alg:radi}. We can write $T = F_1F_1^\TT$ and $\Y_2 = F_2F_2^\TT$, and then we have
\begin{align*}
	\mathscr{H}_2  &=
	\begin{bmatrix}
		\rr_1^2F_1F_1^\TT+\rr_1\rr_2F_2F_2^\TT & \rr_1\rr_2(F_1F_1^\TT-F_2F_2^\TT) \\
		\rr_1\rr_2(F_1F_1^\TT-F_2F_2^\TT)    &  \rr_2^2 F_1F_1^\TT +\rr_1\rr_2F_2F_2^\TT
	\end{bmatrix} \\
	& =: \begin{bmatrix}
		\alpha F_1 & \gamma F_2 \\
		\beta F_1    &  \eta F_2
	\end{bmatrix}
	\begin{bmatrix}
		\overline{\alpha} F_1^\TT & \overline{\beta} F_1^\TT \\
		\overline{\gamma} F_2^\TT    &  \overline{\eta} F_2^\TT
	\end{bmatrix},
\end{align*}
where the second identity holds if
\[
\alpha\overline{\alpha} = \rr_1^2,\quad \gamma\overline{\gamma} =\rr_1\rr_2= \overline{\eta}\eta ,\quad
\beta\overline{\beta} = \rr_2^2
\]
and
\[
\alpha \overline{\beta} = \rr_1\rr_2 = -\gamma\overline{\eta}, \quad
\overline{\alpha} \beta = \rr_1\rr_2 = -\eta\overline{\gamma}.
\]
Since $\rr_1$ and $\rr_2$ are real numbers, it holds that
$\overline{\rr_1\rr_2} = \rr_1\rr_2$ and therefore the set of conditions on the scalar coefficients
$\alpha$, $\beta$, $\gamma$, $\eta$ in the complex domain can be simplified as
\begin{equation*}
	|\alpha| = |\rr_1|,\quad
	|\gamma|^2 =  \rr_1\rr_2,\quad
	|\eta|^2 =\rr_1\rr_2,\quad
	|\beta| = |\rr_2|
\end{equation*}
and
\[
\alpha \overline{\beta} = \rr_1\rr_2 = -\gamma\overline{\eta}.
\]
A solution $\{\alpha, \beta, \gamma, \eta\}$ to this set of equations obviously exists if $\rr_1\rr_2\ge 0$: we can simply take
\[
\alpha = \pm \rr_1,\quad
\beta = \pm \rr_2
\]
and
\[
\gamma = \pm \sqrt{\rr_1\rr_2},\quad
\eta = \mp \sqrt{\rr_1\rr_2}
\]
to obtain four different sets of real solutions. In other words, since $\rr_1+\rr_2=1$, we have shown that if $0\le \rr_1,\rr_2 \le 1$, then the intermediate factor $\mathscr{H}_2$ is positive semi-definite and hence the residual of the extrapolant
$\R(\rr_1X_1 + \rr_2X_2)$ is positive semi-definite.

Next, we show by contradiction the condition $0\le \rr_1,\rr_2 \le 1$ on the other hand is a necessary condition for the residual of the extrapolant to be positive semi-definite.
Now suppose $\R(\rr_1X_1 + \rr_2X_2)$ is positive semi-definite, so the intermediate factor $\mathscr{H}_2$ must be positive semi-definite. Suppose, however, the condition $0\le \rr_1,\rr_2 \le 1$ is not satisfied, then from the relation $\rr_1+\rr_2=1$ it holds that
$\rr_1\rr_2<0$, that is,
$\rr_1 <0, \rr_2>1$ with $|\rr_2|>|\rr_1|$
or $\rr_2 <0, \rr_1>1$ with $|\rr_1|>|\rr_2|$. This is to say,
we have $\rr_1\rr_2<0$ with either $|\rr_1\rr_2|>\rr_1^2$ or $|\rr_1\rr_2|>\rr_2^2$.

As a consequence of positive semi-definiteness of $\mathscr{H}_2$, its $(1,1)$-block (one of the leading principal submatrix)
$\rr_1^2F_1F_1^\TT+\rr_1\rr_2F_2F_2^\TT\equiv \rr_1^2T+\rr_1\rr_2\Y_2$ is positive semi-definite and also $\rr_2^2F_1F_1^\TT+\rr_1\rr_2F_2F_2^\TT\equiv\rr_2^2 T +\rr_1\rr_2\Y_2$ is positive semi-definite (just consider the positive semi-definiteness of the Schur complement).
Without loss of generality, suppose, from the above discussion,
$\rr_1\rr_2<0$ holds with $|\rr_1\rr_2|>\rr_1^2$. Then, we have
\begin{equation*}
	\rr_1^2T+\rr_1\rr_2\Y_2 + \rr_1^2\Y_2 - \rr_1^2\Y_2  \succcurlyeq 0 \quad \Longleftrightarrow \quad
	\rr_1^2(T - \Y_2 )  \succcurlyeq (|\rr_1\rr_2| - \rr_1^2)\Y_2,
\end{equation*}
where the latter Loewner order relation implies that
$\rr_1^2(T - \Y_2 )$ is positive semi-definite because
$ (|\rr_1\rr_2| - \rr_1^2)\Y_2$ is positive semi-definite.
On the other hand,
recall from the modified RADI iteration (Algorithm~\ref{alg:radi}) that, with $\sigma_k<0$, we have
$\Y_2 \succcurlyeq T$, which is equivalent to say $T-\Y_2$ is negative semi-definite; a contradiction!

Therefore, we have proved above that
$0\le \rr_1,\rr_2 \le 1$ is a sufficient and necessary condition for the residual of the extrapolant $\R(\rr_1X_1 + \rr_2X_2)$ to be positive semi-definite.
%\reply{I think the main difficulty is revealed in the proof of the previous theorem: for $n\ge3$
%there are the $C_{ik}$ terms (see~\eqref{eq:alpha-Cik}) that destroy the nice structure for $n=2$.}
\end{proof}

\bibliographystyle{plainurl}
\bibliography{journals, literature}

\end{document}